\documentclass[10pt,a4paper]{amsart}
\usepackage{amsfonts}
\usepackage{mathtools}
\usepackage{amsthm}
\usepackage{amsmath}
\usepackage{amssymb}
\usepackage{amscd}
\usepackage{cite}
\usepackage[ngerman,english]{babel}
\usepackage[mathscr]{eucal}
\usepackage{indentfirst}
\usepackage{graphicx}
\usepackage{graphics}
\numberwithin{equation}{section}
\usepackage[margin=3.5cm]{geometry}
\usepackage{epstopdf} 
\usepackage{color} 
\usepackage{hyperref}

\allowdisplaybreaks 

\theoremstyle{plain}
\newtheorem{Th}{Theorem}[section]
\newtheorem{Lemma}[Th]{Lemma}

\newtheorem{Prop}[Th]{Proposition}
\DeclareMathOperator{\R}{\mathbb{R}}

\DeclareMathOperator{\N}{\mathbb{N}}
\DeclareMathOperator{\di}{\text{div}}
\DeclareMathOperator{\supp}{\text{supp}}

\DeclareMathOperator{\dist}{\text{dist}}
\DeclareMathOperator{\E}{\alpha}

\theoremstyle{definition}

\newtheorem{Corollary}[Th]{Corollary}

\newtheorem{Rem}[Th]{Remark}
\newtheorem{?}[Th]{Problem}

\newcommand{\norm}[2]{\left\lVert #1 \right\rVert_{#2}}
\newcommand{\f}[2]{\frac{#1}{#2}}
\newcommand{\eps}{\varepsilon}

\begin{document}

\title[Biharmonic wave maps]{Biharmonic wave maps: Local wellposedness in high regularity}

\author[S.~Herr, T.~Lamm, T.~Schmid, R.~Schnaubelt]{Sebastian Herr, Tobias Lamm, Tobias Schmid and Roland Schnaubelt}

\address{Fakult\"at f\"ur Mathematik\\ Universit\"at Bielefeld\\Postfach 10 01 31 \\ 33501 Bielefeld  \\ Germany}
\email{herr@math.uni-bielefeld.de}

\address{Department of Mathematics \\ Karlsruhe Institute of
  Technology\\ Englerstra{\ss}e 2\\ 76131 Karlsruhe \\ Germany}
\email{tobias.lamm@kit.edu\\tobias.schmid@kit.edu\\roland.schnaubelt@kit.edu}

\thanks{TL, TS and RS gratefully acknowledge financial support by the Deutsche Forschungsgemeinschaft (DFG)
through CRC 1173.}

%\subjclass[2010]{Primary: 35L75. Secondary: 58J45}

\date{\today}

%\keywords{Biharmonic wave map, energy estimates, high regularity} 

\begin{abstract} 
We show the local wellposedness  of biharmonic wave maps with initial data of sufficiently high Sobolev regularity and
 a blow-up criterion in the sup-norm of the gradient of the solutions. In contrast to the wave maps equation we use 
a vanishing viscosity argument and an appropriate parabolic regularization in order to obtain the existence result. The geometric nature 
of the equation is exploited to prove convergence of approximate solutions, uniqueness of the limit, and continuous 
dependence on initial data.
\end{abstract}

\maketitle

\section{Introduction}
Let $ (N, g)$ be a smooth Riemannian manifold which we assume to be isometrically embedded into some Euclidean space $\R^L$. 
Biharmonic wave maps are critical points $ u : \R^n \times [0, T) \to N$ of the (extrinsic) action functional
\begin{equation}\label{action}
\Phi(u) = \frac12 \int_0^T \int_{\R^n} |\partial_t u|^2 - | \Delta u |^2 \,dx\,ds.
\end{equation} 
These maps model the movement of a thin, stiff, elastic object within the target manifold $N$.

The Euler-Lagrange equation of $\Phi$ has been calculated in \cite{LammSchnaubeltHerr} (in the case $N=S^l\subset \R^{l+1}$) and in \cite{schmid} 
(for arbitrary $N$) and it is given by
 \begin{align}\label{EL-condition}
 \partial_t^2u + \Delta^2u \perp T_u N \qquad \text{on } \R^n \times [0, T).
 \end{align}
In particular, if the manifold $N$ has non-vanishing curvature, the
condition \eqref{EL-condition} is rewritten as a nonlinear partial
differential equation
\begin{align}\label{equation-intro}
\partial_t^2 u + \Delta^2 u = \mathcal{N}(u, u_t, \nabla u, \nabla^2 u, \nabla^3 u ),
\end{align}
where $\mathcal{N}$ is a nonlinear expression of the indicated  derivatives of $u$. It is explicitely given in \eqref{expansion}. 
We also note that the following energy 
$$ E(u) = \f{1}{2} \int_{\R^n} |\partial_tu|^2 + |\Delta u|^2 ~dx $$
is (formally) conserved up to the existence time. %For further details, we refer the reader to the explicit computation of \eqref{EL-condition} in Section \ref{prelim}.

In the flat case $ N = \R^L$ (or any affine subspace), the condition  \eqref{EL-condition} reduces to the free evolution of a system 
of decoupled (or linearly coupled) biharmonic wave equations
\begin{align}\label{equation-intro1}
\partial_t^2u + \Delta^2u = 0,
\end{align}
which appear in the elasticity theory of vibrating plates. Here,
requiring the parametrization of a \emph{thin} plate, the bending
energy of the elastic plate involves integrated curvature terms of the
plate's surface. Hence, in the case of \emph{sufficiently stiff}
material, the potential energy in \eqref{action} is a reasonable
approximation of the elastic energy. We refer to the classical book of
Courant and Hilbert \cite[Chapter 5.6]{courant_hilbert} for more information.

Semi-linear evolution problems related to \eqref{equation-intro1} without a geometric constraint, such as
\begin{align}\label{biharmonic_plate_nonlinear}
\partial_t^2u + \Delta^2u + m u + |u|^{p-1} u = 0,
\end{align}
have been thoroughly studied. For instance, if $ m > 0 $ and $ 1 +
\f{8}{n} <  p < \f{n+4}{n-4} $, global existence and scattering of
solutions of \eqref{biharmonic_plate_nonlinear} has been proved by
Pausader in \cite{pausader}, as conjectured by Levandosky and Strauss.

A well-studied hyperbolic geometric evolution problem is the wave maps equation
\begin{align}\label{wave_map}
\square  u = m^{\alpha \beta} A(u)(\partial_{\alpha} u, \partial_{\beta} u)   \qquad \text{on } \R^n\times\R,
\end{align}
which arises as the Euler-Lagrange equation of a first order analogue of the action  \eqref{action} with constraint $ u \in N $. Here,
$ \square = \partial_t^2 - \Delta $ is the d'Alembert operator, $m$ is the Minkowski metric  and $ A(u) $ is the second fundamental 
form of the embedded manifold $N$. The Cauchy problem for \eqref{wave_map} has been studied intensively
as a model for the subtle interplay of nonlinear dispersion, gauge
invariance and singularity formation. In particular, we refer to the
global regularity theory achieved by novel renormalization techniques
of Tao in \cite{tao1} and \cite{tao2}, see also the survey article by
Tataru \cite{tataru}. In the energy-critical dimension $n=2$ a proof of the
threshold conjecture on the question of blow-up versus global
regularity and scattering is given by Sterbenz and Tataru in \cite{sterbenz-tataru}.  

A different but related model is the Schr\"odinger maps problem for a map  $ u : \R^n \times \R \to N $ into a K\"ahler manifold $N$. This is the Hamiltonian flow for the Dirichlet energy of $u$ induced by the (symplectic) K\"ahler form on $N$. For $ N = S^{2} $ the Hamiltonian equation reads as 
\begin{align}
\partial_t u = u \times \Delta u \qquad  \text{on } \R^n \times \R,
\end{align}
and attracted a lot of attention in the past decades. We refer to the
global regularity results for $ N = S^2$ and $ n \geq 2$ by Bejenaru,
Ionescu, Kenig and Tataru in
\cite{bejenaru_et_al}  and for homogeneous spaces $N$ and large
dimension by Nahmod, Stefanov and Uhlenbeck in \cite{nahmod_et_al}. While different, the
methods in both cases exploit the geometric nature of the Schr\"odinger maps flow by the choice of a suitable frame system along a solution.

\smallskip

In sharp contrast, there is very little literature on the bi-harmonic wave maps \eqref{EL-condition}, as discussed below. 
The main goal of this paper is the proof
of the following local wellposedness result  for the Cauchy problem corresponding to \eqref{EL-condition} 
in Sobolev spaces with sufficiently high regularity. We stress that  it is difficult to employ the energy method for high regularity solutions 
of \eqref{EL-condition} since   $ \mathcal{N}$ explicitly depends on the third order derivatives $ \nabla_x^3 u$ and 
the energy contains only $\Delta u$. We will overcome this difficulty by exploiting the geometric constraints of solutions. From now on, 
let $ N $ be a compact Riemannian manifold,  isometrically embedded into $ \R^L$.

\begin{Th}\label{maintheorem}
	Let $ u_0, u_1 : \R^n \to \R^L$ satisfy $u_0(x) \in N$ and $u_1(x) \in T_{u_0(x)}N$ for a.e.\ $x \in \R^n $ as well as
	$$ (\nabla u_0, u_1) \in H^{k-1}(\R^n) \times H^{k-2}(\R^n) $$ 
	for some $ k \in \N$ with $k > \lfloor \f{n}{2} \rfloor + 2$. Then the following assertions hold:
\begin{itemize}
	\item[(a)]
	There exists a maximal existence time 
	$$T_m =T_m(u_0,u_1) > T = T(\norm{\nabla u_0}{H^{k-1}},\norm{u_1}{H^{k-2}}) >0$$
	 and a unique solution $ u : \R^n \times [0,T_m) \to N $ 
	of \eqref{EL-condition} with $u(0)=u_0$, $\partial_t u(0)=u_1$, and
	$$ u - u_0 \in C^0([0,T_m), H^k (\R^n)) \cap C^1([0,T_m), H^{k-2}(\R^n)).$$
	\item[(b)] For $   T_0 \in (0, T_m)$ there exists a (sufficiently small) radius $ R_0 >0$ such that for all initial data 
	$(v_0, v_1)$ as above that satisfy
	$$ \norm{(u_0, u_1) - (v_0, v_1) }{H^k(\R^n)\times H^{k-2}(\R^n) }  \leq  R_0, $$ 
	the unique solution $ v : \R^n \times[0, T_m(v_0,v_1))\to N $ exists on $  \R^n \times [0, T_0].$ Further, for such initial data
	the map $(v_0,v_1) \to (v(t),\partial_t v(t))$ is continuous in $ H^k(\R^n)\times H^{k-2}(\R^n)$ for $t\in[0,T_0]$.
% 	Further for each
% 	$ \varepsilon > 0 $ there exists $0 < R \leq R_0 $ such that for 
% 	$$ \norm{(u_0, u_1) - (v_0, v_1) }{H^k(\R^n)\times H^{k-2}(\R^n) }  \leq  R$$ 
% 	we have
% 	\begin{align*}
% 	\sup_{ t \in [0, T_0]}\left(\norm{u(t) - v(t) }{H^k(\R^n)} + \norm{u_t(t) - v_t(t) }{H^{k-2}(\R^n)}\right) \leq \varepsilon.
% 	\end{align*}
	\item[(c)] If $T_m<\infty$, then 
\begin{align}
\int_0^{T_m} \norm{\nabla u(s)}{L^{\infty}}^{2k} + \norm{ u_t(s)}{L^{\infty}}^{2k}\,ds = \infty. \label{blowup}
\end{align}
\end{itemize}
In particular, for smooth initial data $ u_0, u_1 : \R^n \to \R^L$ with $u_0(x) \in N$ and $u_1(x) \in T_{u_0(x)}N$ for $x\in\R^n$ having compact 
$ \supp(\nabla u_0) \subset \R^n$ and $\supp(u_1)\subset \R^n$, there exist $ T_m > 0 $ and a smooth solution $ u : \R^n \times  [0, T_m) \to N $ 
of \eqref{EL-condition}.
\end{Th}
It is worthwhile to remark that both $u_0$ and $u(t)$ do not necessarily belong to $L^2(\R^n)$ and it is only the difference of 
these two functions which is  contained in this space. We further mention that the lower bound $ k > \lfloor \f{n}{2} \rfloor + 2 $ ensures the existence of $L^{\infty} $ bounds for $ \partial_t u \in H^{k-2}(\R^n) $ from Sobolev's embedding. This is necessary in order to establish our energy estimates in the following sections. 

The first, second and fourth author have recently shown in \cite{LammSchnaubeltHerr} that there exists a global weak solution 
of \eqref{EL-condition} for initial data in the energy space $H^2\times L^2$ in the case $N=S^l\subset \R^{l+1}$. In \cite{LammSchnaubeltHerr} 
a crucial ingredient is a conservation law which allows to construct the desired solution as a weak limit of a sequence of solutions 
of suitably regularized problems. The derivation of this conservation law relies on the fact that the action functional $\Phi$ 
is invariant under rotations in the highly symmetric setting $N=S^l$, and this argument does not apply to arbitrary target manifolds $N$.

Moreover, the third author has shown energy estimates for biharmonic wave maps in low dimensions $n=1,2$ in \cite{schmid}. 
When combining this result with the above blow-up criterion \eqref{blowup}, he then obtained the 
existence of a unique global smooth solution of \eqref{EL-condition} for smooth and compactly supported initial data. This results extends 
earlier work of Fan and Ozawa \cite{fan2010regularity} for spherical target manifolds.

A local well-posedness result as in Theorem \ref{maintheorem}  is standard for second-order wave equations with derivative nonlinearities
such as wave maps. It can be found for example in the books of Shatah and Struwe \cite{shatahstruwe} and Sogge \cite{sogge}. 
In contrast to this case, our nonlinearity $\mathcal{N}(u)$ depends on the third spatial derivative of $u$ which cannot directly be controlled
by  the energy of \eqref{expansion}  that only contains second order spatial derivatives. In our proof we use the geometric nature of the equation 
in several crucial steps 
in order to be able to rewrite this expression in terms of derivatives of lower order.

Concerning the continuous dependence of the solution on the initial data, as the nonlinearity $ \mathcal{N}(u)$ depends 
on third spatial derivatives, no Lipschitz estimate 
in the norm $H^k \times H^{k-2}$ is expected from the energy method
(as we observe e.g.\ from the a priori estimates in Section \ref{uni})
and we cannot apply a fixed point argument. 
In comparison to semi-linear wave equations with derivative
nonlinearities (such as wave maps), this makes the well-posedness
problem for \eqref{EL-condition} more involved.

We briefly note that our result applies to an intrinsic version of a biharmonic wave map.
The functional $ \Phi$ has an intrinsic analogue $ \Psi$ defined by
\begin{equation}\label{intrinsic-psi} \Psi(u) = \f{1}{2} \int_0^T
  \int_{\R^n} | \partial_t u |^2 - |(\Delta u)^T|^2\,dx\,ds,
\end{equation}
where $ (\Delta u)^T = P_u( \Delta u) $ is the tension field of a smooth function $ u : \R^n \times [0, T)\to N $. In contrast to $ \Phi$, 
the functional $ \Psi$ is independent of the embedding of $N
\hookrightarrow \R^L$.
Since the Euler-Lagrange equation differs only by lower order terms
(see \eqref{intrinsic} in Section \ref{prelim} below),
we can prove the existence of local unique intrinsic biharmonic wave maps with initial data as in Theorem \ref{maintheorem}. 
However, we do not have a result for initial data with (only) \textit{covariant} derivatives in $L^2$.

In the following, we briefly outline the structure of the paper. In Section \ref{ex}, we use a vanishing viscosity 
approximation and solve the corresponding Cauchy  problem for the damped problem
$$ \partial_t^2u + \Delta^2u - \varepsilon \Delta \partial_tu \perp T_uN,~~ \varepsilon \in (0, 1]. $$
In order to obtain a limiting solution for \eqref{EL-condition} as $ \varepsilon\searrow 0$, we prove a priori energy estimates 
which are uniform in $\varepsilon$ in Section \ref{uniform-energy}. As a by-product we obtain the blow-up criterion in Theorem \ref{maintheorem}. 
The existence part in Theorem \ref{maintheorem} is then shown in Section \ref{main-theorem-proof}, and in Section \ref{uni} we prove that the 
solutions are unique. Finally we establish the continuity of the flow map  in Section \ref{conti}. 

\section{Notation and preliminaries}\label{prelim}
In this section and in the following we will write $C$ for a generic constant only depending on $N$, $n$ and $k$, and often also $\lesssim\dots$ instead of $\le C\,(\cdots)$.
In order to obtain the explicit form of \eqref{EL-condition}, we use the fact that there exists some $ \delta_0 > 0 $ and a smooth family 
of linear maps $ P_p : \R^L \to  \R^L $ for $ \dist(p, N) < \delta_0$ such that
$$ P_p : \R^L \to T_pN , \quad p \in N, $$
is an orthogonal projection onto the tangent space $ T_pN $.
The Euler-Lagrange equation \eqref{EL-condition} can thus be written as 
\[
\partial_t^2u + \Delta^2u = (I-P_u)(\partial_t^2u + \Delta^2u) .
\]
Exploiting that $u$ takes values in $N$, we have
\begin{align}\label{expansion}
\partial_t^2u + \Delta^2u &= dP_u(u_t,u_t) + dP_u(\Delta u,\Delta u) + 4 dP_u( \nabla u, \nabla \Delta u) + 2 dP_u(\nabla^2u, \nabla^2 u)\\ \nonumber
&\quad+2 d^2P_u(\nabla u, \nabla u, \Delta u)  + 4 d^2P_u( \nabla u, \nabla u, \nabla^2 u )\\ \nonumber
&\quad+ d^3P_u(\nabla u, \nabla u, \nabla u, \nabla u)\\ \nonumber 
&=: \mathcal{N}(u),
\end{align}
where the tensors $d^jP$ are explicitly described below.

We briefly remark that, compared with the right hand side of \eqref{expansion},
the Euler-Lagrange equation for the intrinsic biharmonic wave maps
problem \eqref{intrinsic-psi} differs by 
\begin{align}\label{intrinsic}
P_u( dP_u(\nabla u, \nabla u) \cdot d^2P_u( \nabla u, \nabla u, \cdot) ) + P_u( \di(dP_u(\nabla u, \nabla u)\cdot dP_u(\nabla u, \cdot))).
\end{align}

The projectors $P_p$ are derivatives of the metric distance (with respect to $N$) in $\R^L$, i.e.,
\begin{align}\label{projectors}
p = \pi(p) + \f{1}{2} \nabla_p ( \dist^2(p, N)),~~ P_p = d_p \pi(p),~~~\dist(p,N) < \delta_0.
\end{align}
Moreover, if $ p \in \R^L $ is sufficiently close to $N$, then $\pi$ has the nearest point property, i.e.,
 $ |\pi(p) - p| = \inf_{q \in N} | q - p|, $ and hence
\begin{align*}
d\pi_{|_p} = d\pi(p) = d(\pi^2 (p)) = d\pi_{|_{\pi(p)}}d\pi_{|_p}.
\end{align*}
Therefore $ P_p : \R^L \to T_{\pi(p)}N $ is well-defined. Using cut-off functions we extend the identity \eqref{projectors}, and thus 
also the equation $P_p = d_p \pi(p)$, to all of $ \R^L$.  Moreover, all derivatives of $P_p$ are bounded on $\R^n$.
In this way one can investigate  \eqref{expansion} without  restricting the coefficients a priori. Further, for $ l \in \N_0$ we denote by 
$ d^lP_p$ the derivative of order $ l $ of the map $ P_p $, which is a $ (l+1)$-linear form on $ \R^L$. For the coefficients 
in the standard coordinates in $ \R^L$ we write
$$  (d^jP_u)^k_{ l_0, \dots, l_j} = \f{\partial}{\partial p_{l_1}}\dots \f{\partial}{\partial p_{l_j}}(P_p)^k_{l_0}(u).$$
We now derive \eqref{expansion} from the condition \eqref{EL-condition} for smooth solutions $ u : \R^m \times [0, T) \to N $. Note that we use the sum convention, i.e. the same indices in super-/subscript means summation.\\[2pt]
Since $ \partial_t u \in T_uN $, we infer the identity 
\begin{align*}
[(I - P_u)(\partial_t^2 u)]^k &=  (\delta^k_l - (P_u)^k_l)(\partial_t^2 u^l) =  \partial_t(\delta^k_{l} - (P_u)^k_{l})(\partial_t u^{l}) +   (\partial_mP_u)^k_{l} \partial_t u^{l}\partial_t u^m\\ \nonumber
& =  (dP_u)^k_{m,l} \partial_t u^{l}\partial_t u^{m}
\end{align*}
 for $ k = 1, \dots, L$. Because of $ \nabla u \in T_uN $, we also obtain
\begin{align*}
[(I - P_u)(\Delta u)]^k &=  \partial^{x_{\alpha}}(\delta^k_{l} - (P_u)^k_{l})(\partial_{x_{\alpha}} u^{l}) +  (\partial_mP_u)^k_{l} \partial^{x_{\alpha}}u^{l}\partial_{x_{\alpha}} u^m\\ \nonumber
& =   (dP_u)^k_{m,l} \partial^{x_{\alpha}} u^{l}\partial_{x_{\alpha}} u^{m},
\end{align*}
and hence 
\begin{align*}
[(I - P_u)(\Delta^2 u)]^k &=    \Delta ( (dP_u)^k_{m,l} \partial^{x_{\alpha}} u^{l}\partial_{x_{\alpha}} u^{m}) +  \partial^{x_{\alpha}}  ( (dP_u)^k_{m,l} \Delta u^{l}\partial_{x_{\alpha}} u^{m})\\
&~~~ +    (dP_u)^k_{m,l} (\partial^{x_{\alpha}}\Delta u^{l})\partial_{x_{\alpha}} u^{m}.
\end{align*}
The symmetry of the indices then implies
\begin{align*}
[(I - P_u)(\Delta^2 u)]^k &= (d^3P_u)^k_{ l_0, l_1, l_2, l_3}\partial^{x_{\alpha}} u^{l_0} \partial_{x_{\alpha}} u^{l_1} \partial^{x_{\beta}} u^{l_2} \partial_{x_{\beta}} u^{l_3}\\
&~~~ + 2  (dP_u)^k_{ l_0, l_1}  \partial_{x_{\alpha} } \partial^{x_{\beta}} u^{l_0} \partial^{x_{\alpha}} \partial_{ x_{\beta}} u^{l_1} +  (dP_u)^k_{l_0, l_1}  \Delta  u^{l_0} \Delta u^{l_1}\\
&~~~  + 2  (d^2P_u)^k_{l_0, l_1, l_2}  \partial^{x_{\alpha}} u^{l_0} \partial_{x_{\alpha}} u^{l_1} \Delta u^{l_2} + 4   (dP_u)^k_{l_0, l_1}  \partial^{x_{\alpha}}\Delta  u^{l_0} \partial_{x_{\alpha}} u^{l_1}\\
&~~~ + 4  (d^2P_u)^k_{l_0, l_1, l_2}  \partial^{x_{\alpha}}  u^{l_0} \partial_{x_{\alpha}}\partial^{x_{\beta}} u^{l_1}  \partial_{x_{\beta}} u^{l_2}.
\end{align*}
We briefly state the expressions from \eqref{intrinsic} in coordinates, i.e.,
\begin{align*}
\left[P_u(dP_u(\nabla u, \nabla u) \cdot d^2P_u( \nabla u, \nabla u, \cdot))\right]^l 
 &=  \sum\nolimits_j (P_u)^l_j dP_u(\nabla u, \nabla u) \cdot (d^2P_u)_{k, m ,j} \partial_{x_{\alpha}} u^k \partial^{x_{\alpha}} u^m,\\
\left[P_u( \di(dP_u(\nabla u, \nabla u)\cdot dP_u(\nabla u, \cdot)))\right]^l 
   &= \sum\nolimits_{j} (P_u)^l_j \partial^{x_{\alpha}}( dP_u(\nabla u, \nabla u) \cdot (dP_u)_{k j} ~\partial_{x_{\alpha}} u^k )
\end{align*}
 for $ l = 1 , \dots , L $.
In the following we use the shorthand $ \nabla^{k_1} u \star  \nabla^{k_2}u $ for (linear combinations of) products of partial derivatives 
of the components $u^l$ of $u$ for $l = 1, \dots, L$. Here the partial derivatives are of order $ k_1 \in \N$ and $ k_2 \in \N $, 
respectively. With this notation we can rewrite equation \eqref{expansion} as
\begin{align*}
\partial_{t}^2 u + \Delta^2u &= dP_u( u_t, u_t) + dP_u(\nabla^2 u \star \nabla^2 u) + dP_u(\nabla^3 u \star \nabla u )\\
&~~~ + d^2P_u(\nabla u \star \nabla u \star \nabla^2 u)  + d^3P_u(\nabla u \star \nabla u \star \nabla u \star \nabla u).
\end{align*}
The Leibniz formula implies the following identity.
\begin{Lemma}\label{Lemma-Leibniz}
	For $ m \in \N$ and $ l \in \N_0$ we have
	\begin{equation} \label{eq:leibniz}
	\nabla^m (d^lP_u) = \sum_{j = 1}^m \sum_{\sum_{k=1}^j m_k = m -j} d^{j+l}P_u(\nabla^{m_1 +1}u \star \dots \star \nabla^{m_j +  1}u).
	\end{equation}
\end{Lemma}

%As mentioned above  
%$$ d^{j+l}P_u( \nabla^{m_1 +1}u \star \dots \star \nabla^{m_j +1} u)$$
%is a $l +1$ linear form. 
In order to include the case $ m = 0 $ in the lemma, we will use $ \sum_{j = \min\{ 1, m\}}^m $ for the sum 
in \eqref{eq:leibniz} or similar  formulas.  The calculation of derivatives $ \nabla^m(\mathcal{N}(u))$ and 
$ \nabla^m( \mathcal{N}(u) - \mathcal{N}(v))$ for sufficiently regular  $ u,v : \R^n \times [0, T] \to \R^L$ and $ m \in \N_0 $ 
has been included in Appendix \ref{section-appendix}, employing the $ \star$-convention. The results from Appendix \ref{section-appendix} 
will be used frequently throughout the paper. In the following sections, we also need a version of the classical Moser estimate, 
see e.g.\ \cite[Chapter 13]{taylor3}. 
\begin{Lemma}\label{Moser-estimate}
	Let $ l,k \in \N$ and $\alpha_1 ,\dots, \alpha_l \in \N_0^n$ satisfy $\sum_{i =1}^l  | \alpha_i | = k$. 
	There exists  $ C > 0 $ such that for all  $ f_1, \dots, f_l \in C_0(\R^n)\cap H^k(\R^n)$ we have
	\begin{align}
	\norm{D^{\alpha_1}f_1 \cdot \dots \cdot D^{\alpha_l}f_l}{L^2} 
	    \leq C \prod_{ i = 1}^l \norm{f_i}{L^{\infty}}^{1 - \f{|\alpha_i|}{k}}  \norm{f_i}{H^k}^{\f{|\alpha_i|}{k}}.
	\end{align}
	In particular,
	\begin{align}
	\norm{D^{\alpha_1}f_1 \cdot \dots \cdot D^{\alpha_l}f_l}{L^2} 
	   \leq C \sum_{j= 1 }^l \prod_{ i \neq j}^l \norm{f_i}{L^{\infty}} \left( \norm{f_1}{H^k} + \dots + \norm{f_l}{H^k}\right).
	\end{align}
\end{Lemma}

\section{Existence for the parabolic approximation}\label{ex}
Since $ \mathcal{N}(u) = \mathcal{N}(u,u_t, \nabla u, \nabla^2 u, \nabla^3 u ),$
 energy estimates for the operator $ \partial_t^2 + \Delta^2 $ are not sufficient to show the existence of a solution of \eqref{expansion}. 
 Instead, we use the damped plate operator 
 $$  \partial_t^2 + \Delta^2 - \varepsilon \Delta \partial_t, $$ 
with $ \varepsilon \in (0, 1] $ fixed, as a regularization. More precisely, we prove the existence of a solution 
$ u^{\varepsilon}: \R^n \times [0, T_{\varepsilon} ] \to N$ of the Cauchy problem
\begin{align}\label{Reg-Cauchy-problem}
\begin{cases}
\partial_t^2u^{\varepsilon}(x,t) + \Delta^2u^{\varepsilon}(x,t) - \varepsilon \Delta \partial_t u^{\varepsilon}(x,t) \perp T_{u^{\varepsilon}(x,t)} N, 
 &  (x,t) \in \R^n\times [0, T_{\varepsilon}], \\[3pt]
u^{\varepsilon}(x,0) = u_0(x),~u^{\varepsilon}_t(0, x) = u_1(x), &  x \in \R^n, 
\end{cases}
\end{align}
where $ u_0, u_1 : \R^n \to \R^L$ satisfy  $u_0(x) \in N$ and $u_1(x) \in T_{u_0(x)}N$ for a.e. $x \in \R^n $ as well as
\[ (\nabla u_0, u_1) \in H^{k-1}(\R^n) \times H^{k-2}(\R^n) \]
for some $ k \in \N$ with $k > \lfloor \f{n}{2} \rfloor + 2 $. In the following we mostly drop the super-/subscript $ \varepsilon$ and write $(u,T)$ 
instead of $ (u^{\varepsilon}, T_{\varepsilon})$. We note that the condition in \eqref{Reg-Cauchy-problem} reads as 
\begin{align}\label{Reg-Cauchy-problem2}
\partial_t^2u + \Delta^2u- \varepsilon \Delta \partial_t u = \mathcal{N}(u) - \varepsilon ( I - P_u)( \Delta \partial_t u).
\end{align}
Using $u(t,x)\in N$, we can expand
\begin{equation}\label{eq:Nep}
\varepsilon ( I -P_u)( \Delta \partial_t u ) =  \varepsilon d^2 P_u ( u_t, \nabla u, \nabla u) + \varepsilon 2 dP_u(\nabla u_t, \nabla u) 
                                               + \varepsilon dP_u( u_t, \Delta u).
\end{equation}
We thus study the regularized problem
\begin{align}\label{Reg-expansion}
\partial_t^2 u + \Delta^2 u - \varepsilon \Delta \partial_tu 
&= \mathcal{N}(u) - \varepsilon d^2 P_u ( u_t, \nabla u, \nabla u) - \varepsilon 2 dP_u(\nabla u_t, \nabla u) - \varepsilon dP_u( u_t, \Delta u)\\ \nonumber
&=: \mathcal{N}_{\varepsilon}(u). 
\end{align}
We next solve \eqref{Reg-expansion}  without the geometric constraint, recalling that only $u(t)-u_0\in L^2(\R^n)$.
\begin{Lemma}\label{Lemma-existence}
	 Let $ \varepsilon \in (0,1) $ and take $ u_0, u_1 : \R^n \to \R^L$ with $u_0(x) \in N$ and $u_1(x) \in T_{u_0(x)}N$ for a.e.\ $x \in \R^n $
	 such that 
	$$ (\nabla u_0, u_1) \in H^{k-1}(\R^n) \times H^{k-2}(\R^n) $$ 
	for some $ k \in \N$ with $k > \lfloor \f{n}{2} \rfloor + 2 $. Then \eqref{Reg-expansion} has a unique local solution 
	$ u : \R^n \times [0, T_\varepsilon] \to \R^L $ satisfying $ u(0) = u_0$, $ u_t(0) = u_1$, and
	\begin{align}
	u - u_0 \in C^0([0,T_\varepsilon], H^k (\R^n)) \cap C^1([0,T_\varepsilon], H^{k-2}(\R^n))\cap  H^1(0,T_\varepsilon; H^{k-1}(\R^n)).
	\end{align}
 In addition,
\begin{align}
	\nabla u \in L^2(0,T_\varepsilon; H^{k}(\R^n))
\end{align}
and there exists a constant $C<\infty$ such that for $ 0 \leq t \leq T_\varepsilon$
\begin{align}\label{Lemma-existence-estimate}
&\norm{\nabla^{k-2}u_t(t)}{L^2}^2+ \norm{\nabla^{k}u(t)}{L^2}^2 + \varepsilon \int_0^t \norm{\nabla^{k-1}u_t(s)}{L^2}^2\,ds+\varepsilon \int_0^t \norm{\nabla^{k+1}u(s)}{L^2}^2\,ds\\\ \nonumber
& \quad\leq C\Big( \int_0^t \int_{\R^n} \nabla^{k-2} (\mathcal{N}_{\varepsilon}(u))\cdot \nabla^{k-2}  u_t\,dx\,ds  + \norm{\nabla u_0}{H^{k-1}}^2 + \norm{u_1}{H^{k-2}}^2\Big).
\end{align}
\end{Lemma}
Before we prove Lemma \ref{Lemma-existence},  we reduce the problem to functions in $L^2$ by setting $ v(x,t) = u(x,t) - u_0(x) $. 
We thus rewrite \eqref{Reg-expansion} as
\begin{align}\label{first-order-system}
\partial_t U + \mathcal{A}_kU = \begin{pmatrix}
0\\
f_{\varepsilon}(U)
\end{pmatrix}, \qquad U(0) = \begin{pmatrix}
0\\
u_1
\end{pmatrix},
\end{align} 
where $ U = \begin{pmatrix}
v\\
v_t
\end{pmatrix} $ and $ f_{\varepsilon}(U) $ is defined through
\begin{align}
f_{\varepsilon}(U) : &= \mathcal{N}( v + u_0)  - \varepsilon d^2 P_{v + u_0} ( v_t, \nabla (v + u_0), \nabla ( v + u_0))\\ \nonumber
&~~~ - \varepsilon 2 dP_{v + u_0}(\nabla v_t, \nabla (v + u_0)) - \varepsilon dP_{v + u_0}( v_t, \Delta ( v + u_0)) - \Delta^2 u_0.
\end{align}
Further the operator $ \mathcal{A}_k : H^k(\R^n) \times H^{k-2}(\R^n) \supseteq \mathcal{D}(\mathcal{A}) \to  H^k(\R^n) \times H^{k-2}(\R^n) $ is 
given by
\begin{align}
\mathcal{A}_k = \begin{pmatrix}
0& -I\\
 \Delta^2 & -\varepsilon \Delta 
\end{pmatrix}, \quad  \mathcal{D}(\mathcal{A}) = H^{k+2}(\R^n) \times H^k(\R^n).
\end{align}
Since the operators $\mathcal{A}_k$ extend each other we drop the subscript $ k $.  It is well known that $-\mathcal{A}$ generates 
an analytic $C^0$-semigroup $ \{S_{\varepsilon}(t)\}_{t \geq 0} $, see e.g.\ \cite[Prop. 2.3]{denkschnaubelt} for the case $ k = 2$. Using also
standard parabolic theory, see e.g.\ \cite[Prop. 0.1]{lasieckacontrol} and \cite[Prop. 1.13]{lunardiinterpolation}, we obtain a first linear existence 
result with some extra regularity.
\begin{Lemma} \label{linear}
	Let $ r \in \N_0$, $ u_1 \in H^{r+1}(\R^n) $, and  $ g \in C^0([0, T],H^r(\R^n) )$. Then there exists a unique solution $ U $ of the linear equation
	\begin{align}\label{linear-equation}
	\partial_t U + \mathcal{A}U = \begin{pmatrix}
	0\\
	g
	\end{pmatrix}, \qquad U(0) = \begin{pmatrix}
	0\\
	u_1
	\end{pmatrix},
	\end{align}
	satisfying
	\begin{align*}
	U \in L^2(0,T; H^{r+4}\times H^{r+2}(\R^n) )\cap C^0(0, T; H^{r+3} \times H^{r+1}(\R^n))\cap H^1(0,T;H^{r+2}\times H^r(\R^n)).
	\end{align*}
\end{Lemma}
We remark that the solution of \eqref{linear-equation} is given by
\begin{align}
U(t) = S_{\varepsilon }(t) \begin{pmatrix}  0\\ u_1\end{pmatrix} + \int_0^t S_{\varepsilon}(t-s) \begin{pmatrix}  0\\ g(s)\end{pmatrix}\,ds.
\end{align}
We quantify the above result by the following higher-order energy estimates.
\begin{Lemma} \label{linear-energy}	
Let $ r \in \N_0$, $g \in C^0([0, T],H^r(\R^n) )$, $u_1 \in H^{r+1}(\R^n) $, and $ u_0 : \R^n \to \R^L$ with $ \nabla u_0 \in H^{r+3}(\R^n)$. 
Then $v$ from Lemma \ref{linear} satisfies 
\begin{align}\label{linear-energy1}
&\norm{v_t(t)}{H^{r+1}}^2 +  \norm{ v(t)}{H^{r+3}}^2 + \varepsilon \int_{0}^T \norm{\nabla v_t(s)}{H^{r+1}}^2\,ds 
  + \varepsilon \int_{0}^T \norm{\nabla (v + u_0)(s)}{H^{r+3}}^2\,ds\\\nonumber
&\quad\leq C (1+ T)\left( \frac{1}{\varepsilon}\int_{0}^T \norm{g(s) + \Delta^2 u_0}{H^{r}}^2 \,ds + \norm{u_1}{H^{r+1}}^2 + \norm{  \nabla u_0}{H^{r+2}}^2\right)
\end{align}
 for $ 0 \leq t \leq  T$, and
\begin{align}\label{linear-energy2}
&\norm{\nabla^{r+1}v_t(t)}{L^2}^2 + \norm{ \nabla^{r+3} v(t)}{L^2}^2 
 + \varepsilon \int_{0}^T \norm{\nabla^{r+2} v_t(s)}{L^2}^2\,ds\\[5pt]\nonumber
&\quad\leq C  \Big( -\int_{0}^t \int_{\R^n} \nabla^{r} \left(g(s) + \Delta^2 u_0\right) \cdot \nabla^{r} \Delta v_t \,dx \,ds  
   + \norm{u_1}{H^{r+1}}^2 + \norm{  \nabla u_0}{H^{r+2}}^2\Big).
\end{align}
\end{Lemma}
\begin{proof}
Writing $U=(v,v_t)$ in  Lemma \ref{linear},  the function $ u = v + u_0 $ fulfills
\begin{align}\label{linear-energy-equation}
\partial_t^2 u + \Delta^2 u - \varepsilon \Delta \partial_t u = g + \Delta^2u_0 
\end{align}
in $ L^2(0, T;H^r(\R^n) )$. We first differentiate \eqref{linear-energy-equation} of order $ \nabla^{l}$ with $ l \in \{ 0, \dots, r \} $.
Testing with $ - \nabla^l \Delta u_t \in L^2_{t,x}$ and integrating by parts in $x$, we derive
\begin{align}\label{above}
\f{d}{dt}\norm{ \nabla^{l+1} u_t(t)}{L^2}^2& + \f{d}{dt}\norm{ \nabla^{l+3} u(t)}{L^2}^2 + \varepsilon \norm{\nabla^{l+2} u_t(t)}{L^2}^2\\ \nonumber
& \leq  \frac{C}{\varepsilon} \norm{\nabla^l (g + \Delta^2 u_0)}{L^2}^2 +  \f{\varepsilon}{2}\norm{\nabla^{l+2}u_t(t)}{L^2}^2,
\end{align}
which makes sense for a.e.\ $t$. (Here and below we use the duality $(H^1, H^{-1}) $ in intermediate steps.) 
We then absorb the last term by the left-hand side and integrate the inequality in $t$.

To control the second summand with $\varepsilon$ in \eqref{linear-energy1}, we test the differentiated version of \eqref{linear-energy-equation} 
 by $ \varepsilon \nabla^{l} \Delta^2 u$. Here we proceed similarly as before, where we integrate the term
$$ \varepsilon \int_0^T \int_{\R^n} \nabla^l\partial_t^2u \cdot \nabla^l\Delta^2 u \,dx\,ds$$
by parts in $t$ and $x$ before aborbing it.

It remains to estimate the $L^2$-norm of $v_t(t)$ and the $H^2$-norm of $v(t)$. These inequalities
follow by testing the equation with $u_t$ and using the fact that
\[ \norm{u-u_0}{L^{\infty}_tL^2} \leq T \norm{u_t}{L^{\infty}_tL^2}.\qedhere\]
\end{proof}

\begin{proof}[Proof of Lemma \ref{Lemma-existence}]
We aim at constructing a solution $ U \in C^0([0,T], H^k \times H^{k-2})$, but due to  $ \Delta^2 u_0 \in H^{k-4} $ we have 
$f_{\varepsilon}(U) \in  C^0([0,T], H^{k-4}) $, which is insufficient for an application of Lemmas \ref{linear} and \ref{linear-energy}
 in a fixed point argument for $v$.
 
We thus approximate $u_0$ by $u_0^{\delta} \in C^{\infty}(\R^n , \R^L) $ for $ \delta > 0 $ such that $\supp(\nabla u_0^{\delta})\subset \R^n $ is compact with 
\begin{align}
u_0^{\delta} \to u_0 \text{ \ a.e. \ \  and \ \ } \nabla u_0^{\delta} \to \nabla u_0~~\text{in}~ H^{k-1}(\R^n)~~\text{as}~~\delta \to 0^+.
\end{align}
Defining $f_{\varepsilon, \delta }$ as above with $ u_0^{\delta}$ instead of $u_0$, we obtain
$ f_{\varepsilon, \delta }(U) \in C^0([0, T], H^{k-3}(\R^n) )$.    For the data
$ (u_0^{\delta},u_1) $ we now prove the existence of a fixed point for the operator $ v \mapsto \mathcal{S}(v)$ defined through
\begin{align}
\begin{pmatrix}
\mathcal{S}(v)\\
\partial_t \mathcal{S} (v)
\end{pmatrix} = S_{\varepsilon} (t) \begin{pmatrix}
0 \\u_1
\end{pmatrix}  + \int_0^tS_{\varepsilon}(t-s)\begin{pmatrix}
0\\
f_{\varepsilon, \delta}(v)
\end{pmatrix} \,ds,%\\[5pt] \nonumber
%\text{where}~~~v \in C^0([0,T], H^k) \cap C^1([0,T], H^{k-2})
\end{align}
%Thus, we define for $ R > 0,~ T \in (0, 1)$ 
which acts on the space
\begin{align*}
\mathcal{B}_R(T) : = \big \{ & v \in C^0([0,T], H^k) \cap C^1([0,T], H^{k-2})~|~ v(0) = 0 ,~ v_t(0) = u_1,\\
& \norm{v}{\mathcal{B}} := \norm{v_t}{L^{\infty}H^{k-2}} + \norm{v}{L^{\infty}L^2 } + \norm{\nabla (v + u_0^{\delta})}{L^{\infty}H^{k-1}} \leq R \big \},
\end{align*}
for parameters $ R > 0$ and  $T \in (0, 1)$ fixed below and the metric given by
\begin{align*}
\norm{v_1 - v_2}{\mathcal{B}(T)} = \norm{v_1 - v_2}{L^{\infty} H^k} + \norm{\partial_t v_1 - \partial_t v_2}{L^{\infty} H^{k-2}},\quad  v_1, v_2 \in \mathcal{B}_R(T).
\end{align*}
Let $ \varepsilon\in (0,1)$ be fixed. We will show that the map
	$$ \mathcal{S}: \mathcal{B}_R(T) \to \mathcal{B}_R(T)$$
	is strictly contractive with respect to $ \norm{\cdot}{\mathcal{B}(T)}$ if we choose $R = R_{\delta}$ and $T = T_{\delta} $ with
\begin{align}\label{def:Tdelta}
	R^k_{\delta} &= 3 (\norm{\nabla u_0^{\delta}}{H^{k-1}} + \norm{u_1}{H^{k-2}} )^k =: 3 R_{0, \delta}^k, \notag\\
	T_{\delta} &= \f{1}{2} \min \left\{ \left(\f{\sqrt[k]{3}-1}{\sqrt[k]{3}}\right)^2 \f{\varepsilon}{\hat{C}^2(1 + 3 R_{0, \delta}^k)^2}, 
	   \f{\varepsilon }{\hat{C}^2 ( 1 + 6 R_{0, \delta}^k)^2} \right \}
\end{align}	   
for a constant $\hat{C}$ depending only on $N$, $n$, and $k$.
To show  this statement, we have to prove  the estimates
\begin{align}\label{fixed-point-estimate1}
	\norm{\mathcal{S}(v)}{\mathcal{B}} &\leq \f{\hat{C}}{\varepsilon^{\f{1}{2}}} T^{\f{1}{2}} ( 1 + \norm{v}{\mathcal{B}}^k)\norm{v}{\mathcal{B}} 
	+ \norm{\nabla u_0^{\delta}}{H^{k-1}} + \norm{u_1}{H^{k-2}},\\ \label{fixed-point-estimate2}
	\norm{\mathcal{S}(v) - \mathcal{S}(\tilde{v}) }{\mathcal{B}(T)} &\leq \f{\hat{C}}{\varepsilon^{\f{1}{2}}} T^{\f{1}{2}} ( 1 + \norm{v}{\mathcal{B}}^k 
	+ \norm{\tilde{v}}{\mathcal{B}}^k) \norm{v - \tilde{v}}{\mathcal{B}(T)}
\end{align}
 for $ v,\tilde{v} \in \mathcal{B}_R(T)$. To employ the inequality \eqref{linear-energy1} for $ r = k-3 $, we need to bound the norms
$$ \norm{\mathcal{N}_{\varepsilon} ( v(t) + u_0^{\delta})}{H^{k-3}}^2~~\text{and}~~ 
    \norm{\mathcal{N}_{\varepsilon} (v(t) + u_0^{\delta}) - \mathcal{N}_{\varepsilon} (\tilde{v}(t) + u_0^{\delta}) }{H^{k-3}}^2$$
    by $C( 1 + \norm{v}{\mathcal{B}}^k)\norm{v}{\mathcal{B}} $ and $C  ( 1 + \norm{v}{\mathcal{B}}^k 
	+ \norm{\tilde{v}}{\mathcal{B}}^k) \norm{v - \tilde{v}}{\mathcal{B}(T)}$, respectively.
This is done by means of Lemma \ref{diff-Lemma} and Corollary \ref{uniqueness-corollary2} combined with a careful application of  the Moser estimate in
Lemma \ref{Moser-estimate}.  We give the relevant  details below in Section \ref{uniform-energy} in the proof of the a priori estimate and 
in Section \ref{uni} for the uniqueness since these parts require more thought.
In this way  we obtain in the fixed point $  v^{\delta} = \mathcal{S}(v^{\delta})$ satisyfying
\begin{align}\label{delta-uniform-bound}
\norm{v^{\delta}_t}{L^{\infty}H^{k-2}}^2+ \norm{ v^{\delta}}{L^{\infty}H^k}^2  + \varepsilon\int_0^{T_{\delta}} \norm{v^{\delta}_t(s)}{H^{k-1}}^2 \,ds 
   + \varepsilon\int_0^{T_{\delta}} \norm{\nabla(v^{\delta} + u_0^{\delta})}{H^{k}}^2 \,ds \lesssim R^2_{\delta}.
\end{align}
In particular, $ v^{\delta} \in L^2(0,T_{\delta}; H^{k+1}) \cap H^1(0,T_{\delta}; H^{k-1} )$. 

We next define $ R_0,$ $R$ and $\tilde{T} > 0 $ in the same way as  $ R_{0, \delta},$ $R_{\delta}$ and $T_{\delta} $ using $ u_0$ instead of $ u_0^\delta$ 
and the $R_0$ instead of $ R_0^\delta$. Thus,
$$ R_{0, \delta} \to R_0, \quad R_{\delta} \to R, \quad T_{\delta} \to \tilde{T} \quad \text{as}~ \delta \to 0^+.$$
For sufficiently small  $ \delta > 0 $ we have $  T_{\delta} > \f{1}{2}\tilde{T} = : T $ and $ | R_{0, \delta} - R_0| \leq R_0 $.
Hence  $ v^{\delta} : \R^n \times [0,T] \to \R^L $ is well defined and $ \norm{v^{\delta}}{\mathcal{B}(T)} \leq C R $ for a constant $C > 0 $.
%We now argue that wlog 
%$ v^{\delta} \to v,~\nabla ( v^{\delta} + u_0^{\delta}) \to \nabla ( v + u_0),  ~ \partial_t v^{\delta} \to \partial_tv $
%strongly in $C^0([0,T], H^k)$,~$ L^2(0,T; H^k)$  and $C^0([0,T], H^{k-2})\cap L^2(0,T; H^{k-1})$, respectively. 
Observe that for sufficiently small $ \delta, \delta' > 0 $, the differences $ v^{\delta} - v^{\delta'}$
and $\partial_t v^{\delta}- \partial_t v^{\delta'} $ solve \eqref{linear-equation} with the nonlinearity 
$$\mathcal{N}_{\varepsilon}(v^{\delta} + u_0^{\delta}) - \mathcal{N}_{\varepsilon}(v^{\delta'} + u_0^{\delta'}) + \Delta^2(u_0^{\delta} - u_0^{\delta'}) \in C^0([0,T],H^{k-3}).$$
Similar to the proof of the Lipschitz estimate \eqref{fixed-point-estimate2}, Lemma \ref{linear-energy} then yields the bound
\begin{align*}
\norm{v^{\delta} - v^{\delta'}}{\mathcal{B}(T)}^2 +&  \varepsilon \int_0^{T_{\delta}} \norm{v^{\delta}_t(s) - v^{\delta'}_t(s) }{H^{k-1}}^2 ds 
     + \varepsilon \int_0^{T_{\delta}} \norm{\nabla(v^{\delta} -v^{\delta'})   + \nabla (u_0^{\delta} - u_0^{\delta'})}{H^{k}}^2 ds\\ \nonumber
&\leq C \f{T}{\varepsilon}(1 + R^{2k})\norm{v^{\delta} - v^{\delta'}}{\mathcal{B}(T)}^2  + \tilde{C}_{\varepsilon, R}\norm{\nabla u_0^{\delta}-\nabla u_0^{\delta'}}{H^{k-1}}^2.
\end{align*}
Hence, if $T=T(\varepsilon)$ is sufficiently small, as $\delta\to0$ the functions $v^{\delta}$ tend to a function
$$ v \in C^0([0,T], H^k)\cap C^1([0,T], H^{k-2}) \cap H^1(0,T; H^{k-1})$$ 
with $ \nabla (v + u_0) \in L^2(0,T; H^k)$, where the limits exist in these spaces. In particular, $ (v, v_t)$ is a solution of \eqref{first-order-system} and 
$ u = v + u_0 $ solves \eqref{Reg-expansion}. Moreover, by  \eqref{linear-energy1} the function $ u^{\delta} = v^{\delta} + u_0^{\delta}$ satisfies
inequality \eqref{Lemma-existence-estimate}, and therefore this estimate also holds for $ u $ since $  u^{\delta}_t \to  u_t$ strongly in $C^0([0,T], H^{k-2})$ and 
$ \mathcal{N}_{\varepsilon}( u^{\delta} ) \to \mathcal{N}_{\varepsilon}(u) $ strongly in $ L^2(0,T; H^{k-2})$ because of 
Corollary \ref{uniqueness-corollary2} and Lemma \ref{Moser-estimate}.

For the uniqueness of $v $, we note that, for  a second solution $ \tilde{v}$, the functions $ w = v - \tilde{v}$ and $w_t = v_t - \tilde{v}_t$ 
solve \eqref{linear-equation} with the nonlinearity $ \mathcal{N}_{\varepsilon }(v + u_0) - \mathcal{N}_{\varepsilon}(\tilde{v} +u_0) \in C^0([0,T], H^{k-3})$. 
 Lemma \ref{linear-energy} then yields  the estimate 
\begin{align}
 \norm{v - \tilde{v}}{\mathcal{B}(T)}^2 \leq C \f{T}{\varepsilon} ( 1 + R^{2k}) \norm{v - \tilde{v}}{\mathcal{B}(T)}^2. 
\end{align}
(Note that $ u_0 $ from the Lemma is different, namely $u_0 = 0$.)
Hence, if $T $ is sufficiently small, we obtain $ v = \tilde{v}$ and thus $ u = v + u_0 $ is unique.
\end{proof}
We next show that the above solution  actually takes values in the target manifold.
\begin{Prop}\label{Corollary-existence}
Let $ \varepsilon \in (0,1)$ and take $ u_0, u_1 : \R^n \to \R^L$ with $u_0(x) \in N$ and $u_1(x) \in T_{u_0(x)}N$ for a.e. $x \in \R^n $ satisfying
$$ (\nabla u_0, u_1) \in H^{k-1}(\R^n) \times H^{k-2}(\R^n) $$ 
for some $ k \in \N$ with $k > \lfloor \f{n}{2} \rfloor + 2 $. Then there exists a maximal existence time $T_{\varepsilon,m}\in(0,\infty]$ and a unique 
solution $u\in  \R^n \times [0,T_{\varepsilon,m}) \to N$ of \eqref{Reg-Cauchy-problem} with $u(0)=u_0$, $\partial_t u(0)=u_1$, 
\[ u-u_0\in C^0([0,T_{\varepsilon,m}), H^k) \cap C^1([0,T_{\varepsilon,m}), H^{k-2})\cap  H^1_{loc}([0,T_{\varepsilon,m}), H^{k-1}(\R^n))\]
and  $\nabla u \in L^2_{loc}([0,T_{\varepsilon,m}), H^{k}(\R^n))$ which satisfies \eqref{Lemma-existence-estimate} for $t\in [0,T_{\varepsilon,m})$.
\end{Prop}
\begin{proof}
Fix $\varepsilon\in(0,1)$.
Let $ u: \R^n \times [0,T] \to \R^L $ be the solution of \eqref{Reg-expansion} constructed in Lemma \ref{Lemma-existence}.
We first show that $ u(x,t) \in N $ for  $ x \in \R^n $ and $ t > 0 $ small enough. Since
$$ C^0([0,T], H^k) \hookrightarrow C^0( \R^n \times [0,T])$$
and $ u_0 \in N $ a.e. on $ \R^n$, there exists a time $ \tilde{T}\in (0, T]$ such that for $t \in [0, \tilde{T}]$ the distance
\begin{align*}
\| \dist(u(t), N)\|_{L^\infty} \leq \sup_{x \in \R^n} | u(x,t) - u_0(x)|\lesssim \norm{ u(t) -  u_0}{H^k}
\end{align*}
is so small that $ \bar{u} = \pi (u) $ is well-defined. We then let $w=\bar u-u$ and we note that $w(0)=\partial_t w(0)=0$. Calculating
\begin{align*}
\partial_t^2 \bar u&= d\pi_u \partial_t^2 u+d^2\pi_u (u_t,u_t),\\
\Delta \bar u_t &= d\pi_u \Delta u_t+d^2\pi_u(\Delta u,u_t)+2d^2\pi_u (\nabla u_t,\nabla u)+d^3\pi_u(\nabla u, \nabla u,u_t),\\
\Delta^2 \bar u&= d\pi_u \Delta^2 u+d^2\pi_u (\Delta u,\Delta u)+4d^2\pi_u(\nabla u,\nabla \Delta u)+2d^2\pi_u (\nabla^2 u, \nabla^2 u)\\
&\quad +2d^3\pi_u (\nabla u,\nabla u,\Delta u)+4d^3\pi_u (\nabla u,\nabla u,\nabla^2 u)\\
&\quad +d^4\pi_u (\nabla u,\nabla u,\nabla u,\nabla u),
\end{align*}
 we conclude that
\begin{align*}
(\partial_t^2+\Delta^2-\eps \Delta \partial_t)w&= d\pi_u \Big((\partial_t^2+\Delta^2-\eps \Delta \partial_t)u\Big)+\mathcal{N}_\eps(u)-\mathcal{N}_\eps(u)\\
  &= d\pi_u (\mathcal{N}_\eps(u))  \in  T_{\bar u}N.
\end{align*}
Next, we note that
\[
w_t=\Big((\pi-I)(u)\Big)_t=(d\pi_{\bar u}-I)u_t \perp T_{\bar u} N.
\]
By testing the above equation for $w$ by $w_t$, it follows
\[
\partial_t \frac12 \int_{\R^n} |w_t|^2\,dx +\partial_t \frac12 \int_{\R^n} |\Delta w|^2\,dx +\eps \int_{\R^n} |\nabla w_t|^2\,dx =0.
\] 
This fact implies that $w_t=0$ and hence $w=0$, which means that $u\in N$.

The claimed uniqueness follows similarly to the end of the proof of Lemma \ref{Lemma-existence}. Finally, we let $T_{\varepsilon,m}\ge \tilde{T}$ be the supremum of times $T'>0$ such that we have 
a solution $u:[0,T']\times \mathbb{R}^n\to N$ of  \eqref{Reg-Cauchy-problem} with $u(0)=u_0$, $\partial_t u(0)=u_1$,  
\[u-u_0\in C^0([0,T'], H^k) \cap C^1([0,T'], H^{k-2})\cap  H^1(0,T'; H^{k-1}(\R^n))\]
and  $\nabla u \in L^2(0,T'; H^{k}(\R^n))$ which satisfies \eqref{Lemma-existence-estimate} on $[0,T']$.
% This fact implies that $w_t=\Delta w = 0$ and therefore
% \[
% \partial_t \int_{\R^n} |\nabla w|^2\,dx \leq 2\|\partial_t w\|_{L^2} \|\Delta w\|_{L^2} =0
% \]
% which in turn shows that $\nabla w= 0$. Altogether this implies $w= 0$ and therefore $u\in N$.
\end{proof}
\begin{Rem}\label{rem:ep-sol}
We remark that up to now we fixed $ \varepsilon \in (0,1)$. Since the constants in the upper bound in estimates such as \eqref{delta-uniform-bound} 
are of order $ O\left(\varepsilon^{-1} \right)$, we have to prove $\varepsilon$ independent estimates in the next section. 
\end{Rem}

\section{The a priori estimate}\label{uniform-energy}
We now prove an a priori estimate for the solution $ u^\varepsilon: \R^n \times [0, T_{\varepsilon,m} ) \to N $ of the equation
\begin{equation}\label{regularization}
\partial_t^2u^\varepsilon + \Delta^2u^\varepsilon - \varepsilon \Delta \partial_t u^\varepsilon \perp T_{u^\varepsilon} N \quad \text{on}~ \R^n\times [0, T_{\varepsilon,m})
\end{equation}
given by Proposition \ref{Corollary-existence} with $ \varepsilon \in (0,1)$ and initial data $ u_0, u_1 : \R^n \to \R^L$ such that $u_0(x) \in N$ and
$u_1(x) \in T_{u_0(x)}N$ for a.e. $x \in \R^n $ as well as 
$$ (\nabla u_0, u_1) \in H^{k-1}(\R^n) \times H^{k-2}(\R^n) $$ 
for some $ k \in \N$ with $k > \lfloor \f{n}{2} \rfloor + 2 $.
As before we write  $u$ instead of  $u^{\varepsilon}$, and we fix a number $T<T_{\varepsilon,m}$.
Moreover, \eqref{Lemma-existence-estimate} says that 
\begin{align}\label{energy-estimate}\nonumber
&\norm{\nabla^{k-2}u_t(t)}{L^2}^2 + \norm{\nabla^k u(t)}{L^2}^2 + \varepsilon \int_0^t \norm{\nabla^{k-1}u_t(s)}{L^2}^2 ds\\ 
 &\qquad \lesssim\int_0^t \int_{\R^n} \nabla^{k-2}\left[\mathcal{N}(u) - \varepsilon (I-P_u)(\Delta u_t) \right]\cdot \nabla^{k-2}u_t \,dx \,ds
 + \norm{\nabla^{k-2}u_1}{L^2}^2 + \norm{\nabla^k u_0}{L^2}^2
\end{align}
for $ t \in [0, T]$. We recall that the summand with $\varepsilon$ on the right-hand side is well defined because of \eqref{eq:Nep}.

In the following, we often make use of the relations $ \mathcal{N}(u) \perp T_uN $ and $ u_t \in T_uN $ which hold since $ u(x,t) \in N $ for a.e. 
$(x,t) \in \R^n \times [0, T]$. In particular,   $ \mathcal{N}(u) = (I-P_u)\mathcal{N}(u)$. Using this fact, we first write
\begin{align}\label{eq:cancel}
\nabla^{k-2}( \mathcal{N}(u) ) \nabla^{k-2}u_t &= \sum_{\substack{ m_1 + m_2 = k-2\\ m_1 > 0 }} \nabla^{m_1}( I- P_u) \star \nabla^{m_2}( \mathcal{N}(u)) \nabla^{k-2} u_t \\
&\quad+ \nabla^{k-2}(\mathcal{N}(u)) (I- P_u) \nabla^{k-2}u_t\notag\\
&= \sum_{\substack{ m_1 + m_2 = k-2\\ m_1 > 0 }} \nabla^{m_1}( I- P_u) \star \nabla^{m_2}( \mathcal{N}(u)) \nabla^{k-2} u_t \notag\\
&\quad - \sum_{\substack{ l_1 + l_2 = k-2\\ l_1 > 0 }}\nabla^{k-2}(\mathcal{N}(u))\star \nabla^{l_1}[(I- P_u)] \nabla^{l_2}u_t\notag\\
& =: I_1 + I_2,\notag
\end{align}
where the second equality follows from the Leibniz formula
\begin{align}\label{Leibniz}
0 = \nabla^{k-2} \left[(I-P_u)u_t \right] = \sum_{\substack{ l_1 + l_2 = k-2\\ l_1 > 0 }} \nabla^{l_1}[ (I-P_u)] \star\nabla^{l_2}u_t + (I-P_u) \nabla^{k-2} u_t.
\end{align}
In \eqref{energy-estimate} we thus split 
\begin{align}
\int_{\R^n}\nabla^{k-2}&(\mathcal{N}(u) - \varepsilon (I-P_u)(\Delta u_t)) \cdot \nabla^{k-2}u_t\,dx \notag\\ 
&= \int_{\R^n}\nabla^{k-2}(\mathcal{N}(u))\cdot \nabla^{k-2}u_t\,dx -   \varepsilon\int_{\R^n}\nabla^{k-2}( (I-P_u)(\Delta u_t))\cdot \nabla^{k-2}u_t\,dx\notag \\ 
&= \int_{\R^n} I_1 \,dx + \int_{\R^n} I_2 \,dx - \varepsilon\int_{\R^n}\nabla^{k-2}( (I-P_u)(\Delta u_t))\cdot \nabla^{k-2}u_t\,dx, \label{eq:energy-ep} 
\end{align}

We start by estimating 
$$ \int_{\R^n}I_1 \,dx \leq \sum_{\substack{m_1 + m_2 = k - 2\\ m_1 > 0 }}\norm{\nabla^{m_1}(I-P_u)\star\nabla^{m_2}(\mathcal{N}(u))}{L^2}\norm{\nabla^{k-2}u_t}{L^2}.$$
Lemma \ref{Lemma-Leibniz} yields the identity
\begin{align}\label{eq:nabla-I-Pu}
\nabla^{m_1} (I-P_u) = - \sum_{j = 1}^{m_1} \sum_{\sum_{i=1}^j \tilde{k}_i=m_1-j} d^jP_u( \nabla^{\tilde{k}_1 + 1}u\star\dots\star\nabla^{\tilde{k}_j+1}u),
\end{align}
which implies the pointwise inequality
\begin{align}
|\nabla^{m_1} (I-P_u)| \lesssim \sum_{j = 1}^{m_1} \sum_{\sum_{i=1}^j \tilde{k}_i = m_1-j}|\nabla^{\tilde{k}_1+1}u|\cdots |\nabla^{\tilde{k}_j+1}u|.
\end{align}
On the other hand, Lemma \ref{diff-Lemma} allows us to bound  $ | \nabla^{m_2}( \mathcal{N}(u))| $ pointwise (up to a constant) by terms of the form
\begin{align}\label{eins}
&|\nabla^{\tilde{m}_1 + 1}u | \cdots |\nabla^{\tilde{m}_i +1}u| \, \big[| \nabla^{k_1 }u_t| | \nabla^{k_2}u_t| + | \nabla^{k_1 +2 }u| | \nabla^{k_2+2}u| 
   +| \nabla^{k_1 + 3 }u| | \nabla^{k_2+1}u|  \big],\\ \label{zwei}
& |\nabla^{\tilde{m}_1 + 1}u | \cdots |\nabla^{\tilde{m}_i +1}u \, \big[| \nabla^{k_1 +1 }u| | \nabla^{k_2 + 1}u|| \nabla^{k_3 +2 }u| \big], \\ \label{drei}
& |\nabla^{\tilde{m}_1 + 1}u | \cdots |\nabla^{\tilde{m}_i +1}u|\, \big[| \nabla^{k_1 +1 }u| | \nabla^{k_2 + 1}u|| \nabla^{k_3 +1 }u|| \nabla^{k_4 +1 }u|  \big],
\end{align}
where $  i = 1, \dots, m_2 $ and $ \tilde{m}_1 + \dots +  \tilde{m}_i +  k_1 +  \dots  = m_2 - i$ are as in Lemma \ref{diff-Lemma}. Moreover, in the case $i=0$
(where no derivatives fall on the coefficients) the terms are of the form
\begin{align*}
&| \nabla^{k_1 }u_t| | \nabla^{k_2}u_t| + | \nabla^{k_1 +2 }u| | \nabla^{k_2+2}u| +| \nabla^{k_1 + 3 }u| | \nabla^{k_2+1}u|, \\
&| \nabla^{k_1 +1 }u| | \nabla^{k_2 + 1}u|| \nabla^{k_3 +2 }u|,\\
& | \nabla^{k_1 +1 }u| | \nabla^{k_2 + 1}u|| \nabla^{k_3 +1 }u|| \nabla^{k_4 +1 }u|, 
\end{align*}
where $k_j\in\mathbb{N}_0$ and $k_1 + k_2 + \dots  = m_2$. Note  that $ m_2 \leq k-3$ since $ m_1 > 0$. 
In the following we use the notation \eqref{eins} - \eqref{drei} for all five cases, setting $ i = 0$ for the latter three. 

Combining the above considerations with  Lemma \ref{Moser-estimate}, we can now estimate the norm 
$$ \norm{\nabla^{m_1}(I-P_u)\nabla^{m_2}(\mathcal{N}(u))}{L^2},$$
where we distinguish five cases according to the terms in the brackets in \eqref{eins} - \eqref{drei}.

\smallskip

\noindent
\textit{Case 1:}~ $\nabla^{k_1}u_t \star \nabla^{k_2}u_t$

We use Lemma \ref{Moser-estimate} with
$$ f_1 = \nabla u,~ \dots,~ f_j = \nabla u,~ f_{j+1} = \nabla u,~\dots,~ f_{j + i} = \nabla u,~ f_{j + i +1} = u_t,~f_{j + i+2} = u_t,$$
and derivatives of order 
$$ \tilde{k}_1 + \dots + \tilde{k}_j + \tilde{m}_1 + \dots + \tilde{m}_i + k_1 + k_2  = m_1 + m_2 -i - j = k-2 - (i+j). $$
Employing also  Young's inequality, it follows 
\begin{align*}%\label{final1}
&\norm{| \nabla^{\tilde{k}_1 +1}u | \cdots  |\nabla^{ \tilde{k}_j +1}u||\nabla^{\tilde{m}_1 + 1}u | \cdots |\nabla^{\tilde{m}_i+1}u|| \nabla^{k_1 }u_t| | \nabla^{k_2}u_t }{L^2}\\ \nonumber
&\quad\lesssim \big( (1+\norm{\nabla u}{L^{\infty}}^{k-3})\norm{u_t}{L^{\infty}}^2 +(1+ \norm{\nabla u}{L^{\infty}}^{k-2})\norm{u_t}{L^{\infty}}\big)
( \norm{\nabla u}{H^{k-2 -i-j}} +  \norm{ u_t}{H^{k-2 -i-j}})\\\nonumber
&\quad\lesssim ( 1 + \norm{\nabla u}{L^{\infty}}^{k-1}+ \norm{u_t}{L^{\infty}}^{k-1})( \norm{\nabla u}{H^{k-1}} +  \norm{ u_t}{H^{k-2}}).
\end{align*}
The other cases will be treated similarly. Note that here and in the following the $ L^{\infty}$ norms and especially $\norm{ u_t }{L^{\infty}}$ are bounded by our choice of $ k$.

\smallskip

\noindent
\textit{Case 2:}~ $\nabla^{k_1+2}u \star \nabla^{k_2+2}u$

Here it is exploited that $m_1>0$ in $I_1$ due to the cancellation from \eqref{Leibniz}. 
This time Lemma \ref{Moser-estimate} is applied with $f_1= \dots = f_{j + i+2} = \nabla u$
%$$ f_1 = \nabla u,~ \dots,~ f_j = \nabla u,~ f_{j+1} = \nabla u,~\dots,~ f_{j + i} = \nabla u,~ f_{j + i +1} = \nabla u,~f_{j + i+2} = \nabla u,$$
and derivatives of order 
$$ \tilde{k}_1 + \dots + \tilde{k}_j + \tilde{m}_1 + \dots + \tilde{m}_i + k_1  + k_2 +2   = m_1 + m_2 +2  -i - j = k - (i+j) \leq k-1 $$
since $j > 0$ by \eqref{eq:nabla-I-Pu}.  We estimate
\begin{align*}%\label{final2}
&\norm{| \nabla^{\tilde{k}_1 +1}u | \cdots  |\nabla^{ \tilde{k}_j +1}u||\nabla^{\tilde{m}_1 + 1}u | \cdots 
    |\nabla^{\tilde{m}_i+1}u|| \nabla^{k_1+2 }u| | \nabla^{k_2+2}u| }{L^2}\\
&\quad \lesssim \sum_{i,j}\norm{\nabla u}{L^{\infty}}^{i + j +1}\norm{\nabla u}{H^{k -i-j}}
 \lesssim ( 1+ \norm{\nabla u}{L^{\infty}}^{k-1})\norm{\nabla u}{H^{k -1}}.
\end{align*}

\noindent
\textit{Case 3:}~ $\nabla^{k_1+3}u \star \nabla^{k_2+1}u$

As in  the previous case,  $C(1+  \norm{\nabla u}{L^{\infty}}^{k-1}) \norm{\nabla u}{H^{k-1}}$ dominates
% Again the cancellation from \eqref{Leibniz} is necessary in this case. We use Lemma \ref{Moser-estimate} with
% $$ f_1 = \nabla u,~ \dots,~ f_j = \nabla u,~ f_{j+1} = \nabla u,~\dots,~ f_{j + i} = \nabla u,~ f_{j + i +1} = \nabla u,~f_{j + i+2} = \nabla u,$$
% and derivatives of order 
% $$ \tilde{k}_1 + \dots + \tilde{k}_j + \tilde{m}_1 + \dots + \tilde{m}_i + k_1 + 2  + k_2  = m_1 + m_2 + 2  -i - j = k  - (i+j) \leq k-1 $$
% again since $ j \geq 1$ from $ m_1 > 0 $. Thus 
\begin{align*}%\label{final3}
&\norm{| \nabla^{\tilde{k}_1 +1}u | \cdots  |\nabla^{ \tilde{k}_j +1}u||\nabla^{\tilde{m}_1 + 1}u | \cdots 
 |\nabla^{\tilde{m}_i+1}u|| \nabla^{k_1+3 }u| | \nabla^{k_2+1}u| }{L^2}. %\lesssim(1+  \norm{\nabla u}{L^{\infty}}^{k-1}) \norm{\nabla u}{H^{k-1}}.
\end{align*}

\noindent
\textit{Case 4:}~ $\nabla^{k_1+1}u \star \nabla^{k_2+1}u \star \nabla^{k_3+2}u$

We apply Lemma \ref{Moser-estimate}  to the functions $f_1 = \dots  =f_{j + i+3} = \nabla u$ with 
%$$ f_1 = \nabla u,~ \dots,~ f_j = \nabla u,~ f_{j+1} = \nabla u,~\dots,~ f_{j + i} = \nabla u,~ f_{j + i +1} = \nabla u,~f_{j + i+2} = \nabla u,~ f_{j + i+3} = \nabla u$$
derivatives of order 
$$ \tilde{k}_1 + \dots + \tilde{k}_j + \tilde{m}_1 + \dots + \tilde{m}_i + k_1   + k_2 + k_3 +1 = m_1 + m_2 +1    -i - j = k -1- (i+j), $$
leading to  the bound
\begin{align*}%\label{final4}
&\norm{| \nabla^{\tilde{k}_1 +1}u | \cdots  |\nabla^{ \tilde{k}_j +1}u||\nabla^{\tilde{m}_1 + 1}u | \cdots
  |\nabla^{\tilde{m}_i+1}u|| \nabla^{k_1+1 }u| | \nabla^{k_2+1}u| | \nabla^{k_3+2}u|  }{L^2}\\
&\quad \lesssim  \sum_{i,j}\norm{\nabla u}{L^{\infty}}^{i +j+2} \norm{\nabla u}{H^{k-2 -i-j}}
 \lesssim ( 1 + \norm{\nabla u}{L^{\infty}}^{k}) \norm{\nabla u}{H^{k-1}}.
\end{align*}

\noindent
\textit{Case 5:}~ $\nabla^{k_1+1}u \star \nabla^{k_2+1}u \star \nabla^{k_3+1}u \star \nabla^{k_4+1}u$

We now use Lemma \ref{Moser-estimate} with $f_1 = \dots  =f_{j + i+4} = \nabla u$
% \begin{align*}
% &f_1 = \nabla u,~ \dots,~ f_j = \nabla u,~ f_{j+1} = \nabla u,~\dots,~ f_{j + i} = \nabla u,~f_{j + i +1} = \nabla u,\\
% &f_{j + i+2} = \nabla u,~ f_{j + i+3} = \nabla u,~ f_{j + i+4} = \nabla u
% \end{align*}
and derivatives of order 
$$ \tilde{k}_1 + \dots + \tilde{k}_j + \tilde{m}_1 + \dots + \tilde{m}_i + k_1   + k_2 + k_3 + k_4  = m_1 + m_2    -i - j = k -2- (i+j). $$
Hence, we have
\begin{align*}%\label{final5}
&\norm{| \nabla^{\tilde{k}_1 +1}u | \cdots  |\nabla^{ \tilde{k}_j +1}u||\nabla^{\tilde{m}_1 + 1}u | \cdots |\nabla^{\tilde{m}_i+1}u|| \nabla^{k_1+1 }u| | \nabla^{k_2+1}u| | \nabla^{k_3+1}u| | \nabla^{k_4+1}u| }{L^2}\\\nonumber
&\quad \lesssim  \sum_{i,j}\norm{\nabla u}{L^{\infty}}^{i + j +3}\norm{\nabla u}{H^{k-2 -i-j}} 
 \lesssim ( 1 + \norm{\nabla u}{L^{\infty}}^{k+1}) \norm{\nabla u}{H^{k-1}}.
\end{align*}
Summing up the five cases, we infer
\begin{equation}\label{est:I1}
\|I_1\|_{L^1} \lesssim ( 1 + \norm{\nabla u}{L^{\infty}}^{k-1}+ \norm{u_t}{L^{\infty}}^{k-1})( \norm{\nabla u}{H^{k-1}} +  \norm{ u_t}{H^{k-2}}).
\end{equation}

\smallskip

Next, in $I_2$ from \eqref{eq:energy-ep} we integrate by parts in order to conclude
\begin{align*}\int_{\R^n} I_2 \, dx 
%&= \sum_{\substack{l_1 + l_2 = k-2\\ l_1 > 0}} \int_{\R^n}\nabla^{k-3}( \mathcal{N}(u))\cdot 
% [ \nabla^{l_1 + 1}( I - P_u) \nabla^{l_2}u_t + \nabla{l_1 } (I-P_u) \nabla^{l_2 +1}u_t] \,dx\\
&= \sum_{\substack{l_1 + l_2 = k-2\\ l_1 > 0}} \int_{\R^n} \nabla^{k-3}( \mathcal{N}(u))\star [ \nabla^{l_1 + 1}( I - P_u) \nabla^{l_2}u_t]\,dx\\
&\quad + \sum_{\substack{l_1 + l_2 = k-2\\ l_1 > 0}} \int_{\R^n} \nabla^{k-3}( \mathcal{N}(u))\star [ \nabla^{l_1 }( I - P_u) \nabla^{l_2+1}u_t]\,dx\\
&=: I_2^1 + I_2^2.
\end{align*}
These terms are estimated by 
\begin{align}
|I_2^1| \lesssim \sum_{\substack{l_1 + l_2 = k-2\\ l_1 > 0}} \norm{\nabla^{k-3}(\mathcal{N}(u))}{L^2}\norm{\nabla^{l_1 + 1}( I - P_u) \nabla^{l_2}u_t}{L^2},\\
|I_2^2| \lesssim \sum_{\substack{l_1 + l_2 = k-2\\ l_1 > 0}} \norm{\nabla^{k-3}(\mathcal{N}(u))}{L^2}\norm{\nabla^{l_1 }( I - P_u) \nabla^{l_2+1}u_t}{L^2}.
\end{align}
We control $\norm{\nabla^{k-3}(\mathcal{N}(u))}{L^2}$ by terms of the form \eqref{eins} - \eqref{drei} in the $ L^2 $ norm, obtaining as above
%Then, estimating these norms using Lemma \ref{Moser-estimate} is very similar to the case by case analysis above and we note the bound 
\begin{align*}%\label{final6}
\norm{\nabla^{k-3}(\mathcal{N}(u))}{L^2} \lesssim ( 1 + \norm{\nabla u}{L^{\infty}}^k + \norm{u_t}{L^{\infty}}^{k-2})( \norm{\nabla u}{H^{k-1}} + \norm{u_t}{H^{k-2}}).
\end{align*}
Equation \eqref{eq:nabla-I-Pu} and  Lemma \ref{Moser-estimate} further imply 
%Thus, it remains to estimate (again the cancellation \eqref{Leibniz} is important here)
\begin{align*}%\label{final7}
\norm{\nabla^{l_1 + 1}( I - P_u) \nabla^{l_2}u_t}{L^2} 
&\lesssim  \sum_{j = 1}^{l_1+1} \sum_{\sum_{i=1}^j \tilde{m}_i = l_1+1-j}  \norm{|\nabla^{\tilde{m}_1 + 1}u | \cdots |\nabla^{\tilde{m}_i+1}u|
   | \nabla^{l_2}u_t|}{L^2}\\\nonumber
& \lesssim ( 1 + \norm{\nabla u}{L^{\infty}}^{k-1} + \norm{u_t}{L^{\infty}}^{k-1})( \norm{\nabla u}{H^{k-1}} + \norm{u_t}{H^{k-2}}) 
\end{align*}
where  $ \tilde{m}_1 + \dots + \tilde{m}_i +l_2 = k - 1  -i  \leq k -2$. Similarly, we have
\begin{align*}%\label{final8}
\norm{\nabla^{l_1 }( I - P_u) \nabla^{l_2+1}u_t}{L^2} 
  &\lesssim\sum_{j=1}^{l_1} \sum_{\sum_{i=1}^j \tilde{m}_i = l_1-j} \norm{|\nabla^{\tilde{m}_1 + 1}u | \cdots |\nabla^{\tilde{m}_i+1}u|| \nabla^{l_2+1}u_t|}{L^2}\\\nonumber
& \lesssim ( 1 + \norm{\nabla u}{L^{\infty}}^{k-2} + \norm{u_t}{L^{\infty}}^{k-2})( \norm{\nabla u}{H^{k-1}} + \norm{u_t}{H^{k-2}}) 
\end{align*}
by Lemma \ref{Moser-estimate} with $ \tilde{m}_1 + \dots + \tilde{m}_i +l_2 + 1 = k - 1  -i  \leq k -2,$ since $l_1>0$.
The above three inequalities  yield
\begin{align}\label{est:I2}
\norm{I_2 }{L^1} \lesssim ( 1 + \norm{\nabla u}{L^{\infty}}^{2k-1} + \norm{u_t}{L^{\infty}}^{2k-1})( \norm{\nabla u}{H^{k-1}}^2 + \norm{u_t}{H^{k-2}}^2).
\end{align}

Finally, for the regularization term, we observe
\begin{align*}
- \varepsilon \int_{\R^n} \nabla^{k-2}[ ( I- P_u)(\Delta u_t)] \nabla^{k-2}u_t\,dx &=~~\varepsilon \int_{\R^n} \nabla^{k-3}[ ( I- P_u)(\Delta u_t)]\nabla^{k-1} u_t\,dx\\
&\leq C \norm{\nabla^{k-3}[ ( I- P_u)(\Delta u_t)]}{L^2}^2 + \f{\varepsilon}{2}\norm{\nabla^{k-1}u_t}{L^2}^2.
\end{align*}
In view of \eqref{eq:Nep},  to bound  $ \norm{\nabla^{k-3}[ ( I- P_u)(\Delta u_t)]}{L^2}^2 $  it suffices to estimate  
\begin{align}
&\norm{|\nabla^{\tilde{m}_1 + 1}u | \cdots |\nabla^{\tilde{m}_i + 1}u| \big[| \nabla^{k_1 + 1 }u_t| | \nabla^{k_2 +1} u| + | \nabla^{k_1  }u_t| | \nabla^{k_2+2}u| \big]}{L^2}^2,\\
&\norm{|\nabla^{\tilde{m}_1 + 1}u | \cdots |\nabla^{\tilde{m}_i + 1}u|| \nabla^{k_1 }u_t| | \nabla^{k_2 + 1}u|| \nabla^{k_3 +1 }u|}{L^2}^2,
\end{align}
where $ \tilde{m}_1 + \dots + \tilde{m}_i + k_1 + k_2 +1  = k - 2 - i$ and $ \tilde{m}_1 + \dots + \tilde{m}_i + k_1 + k_2 + k_3  = k - 3 - i$, respectively.
As before, Lemma \ref{Moser-estimate} implies the inequalities
\begin{align}\label{final9}
&\norm{|\nabla^{\tilde{m}_1 + 1}u | \cdots |\nabla^{\tilde{m}_i + 1}u| \big[| \nabla^{k_1 + 1 }u_t| | \nabla^{k_2 +1} u| + | \nabla^{k_1  }u_t| | \nabla^{k_2+2}u| \big]}{L^2}^2\\ \nonumber
&\quad \lesssim ( 1 + \norm{\nabla u}{L^{\infty}}^{2(k-2)} + \norm{u_t}{L^{\infty}}^{2(k-2)})( \norm{u_t}{H^{k-2}}^2 + \norm{\nabla u}{H^{k-2}}^2, \\
\label{final10}
&\norm{|\nabla^{\tilde{m}_1 + 1}u | \cdots |\nabla^{\tilde{m}_i + 1}u|| \nabla^{k_1 }u_t| | \nabla^{k_2 + 1}u|| \nabla^{k_3 +1 }u|}{L^2}^2\\ \nonumber
&\quad \lesssim ( 1 + \norm{\nabla u}{L^{\infty}}^{2(k-1)} + \norm{u_t}{L^{\infty}}^{2(k-1)})( \norm{u_t}{H^{k-2}}^2 + \norm{\nabla u}{H^{k-2}}^2). 
\end{align}
Putting together \eqref{est:I1},  \eqref{est:I2}, \eqref{final9} and \eqref{final10}, we arrive at the inequality
\begin{align*}
&\left| \int_{\R^n}\nabla^{k-2}(\mathcal{N}(u) - \varepsilon (I-P_u)(\Delta u_t)) \cdot \nabla^{k-2}u_t\,dx\right|\\
&\quad \lesssim ( 1 + \norm{\nabla u}{L^{\infty}}^{2k} + \norm{u_t}{L^{\infty}}^{2k})( \norm{\nabla u}{H^{k-1}}^2 + \norm{u_t}{H^{k-2}}^2) 
   + \f{\varepsilon}{2} \norm{\nabla^{k-1} u_t}{L^2}^2.
\end{align*}
Subtracting the last term
%$$ \f{\varepsilon}{2} \int_0^t  \norm{\nabla^{k-1} u_t(s)}{L^2}^2\,ds, $$
on both sides of \eqref{energy-estimate}, for $ t \in [0, T]$ we conclude
\begin{align}\label{est:k}
&\norm{\nabla^{k-2}u_t(t)}{L^2}^2 + \norm{\nabla^k u(t)}{L^2}^2 + \f{\varepsilon}{2} \int_0^t \norm{\nabla^{k-1}u_t(s)}{L^2}^2\,ds \nonumber\\
&\quad \lesssim \int_0^t \Big[( 1 + \norm{\nabla u}{L^{\infty}}^{2k} + \norm{u_t}{L^{\infty}}^{2k})( \norm{\nabla u}{H^{k-1}}^2 
  + \norm{u_t}{H^{k-2}}^2)\Big]ds
  + \norm{\nabla^{k-2}u_1}{L^2}^2 + \norm{\nabla^k u_0}{L^2}^2.
\end{align}

  It remains to bound the lower order terms.  Testing \eqref{regularization} by $u_t \in T_uN $, we infer
\begin{align}\label{est:2}
\norm{u_t(t)}{L^2}^2 + \norm{\Delta  u(t)}{L^2}^2 + \varepsilon \int_0^t \norm{\nabla u_t(s)}{L^2}^2\,ds 
   =  \norm{u_1}{L^2}^2 + \norm{ \Delta  u_0}{L^2}^2.
\end{align}
Since also
$$ \f{d}{dt} \int_{\R^n} |\nabla u|^2 \,dx \leq \int_{\R^n} | u_t|^2\,dx + \int_{\R^n} |\Delta  u|^2\,dx, $$
it follows
\begin{align}\label{est:1}
%\norm{u_t(t)}{L^2}^2 + \|\nabla u(t)\|_{L   2}^2 + \norm{\Delta  u(t)}{L^2}^2 + \varepsilon \int_0^t \norm{\nabla u_t(s)}{L^2}^2\,ds 
%   \le   \norm{u_1}{L^2}^2 + \|\nabla u_0 \|_{L^2}^2 + \norm{ \Delta  u_0}{L^2}^2
 \|\nabla u(t)\|_{L^2}^2 &\le \|\nabla u_0 \|_{L^2}^2 + \int_0^t \norm{\Delta  u(s)}{L^2}^2  + \norm{u_t(s)}{L^2}^2  ds\\
 &= \|\nabla u_0 \|_{L^2}^2 + t(  \norm{u_1}{L^2}^2 + \norm{ \Delta  u_0}{L^2}^2)\notag
\end{align}
for $ t \in [0,T]$. The other derivatives are treated via interpolation, more precisely
\begin{align*}
\norm{\nabla^l u_t}{L^2}^2 &\lesssim \norm{\nabla^{k-1}u_t}{L^2}^{\f{2(l-1)}{k-2}}\norm{\nabla u_t}{L^2}^{\f{2(k-1-l)}{k-2}},\quad l = 2,\dots, k-2,\\
\norm{\nabla^l u_t}{L^2}^2 &\lesssim \norm{\nabla^{k-2}u_t}{L^2}^{\f{2l}{k-2}}\norm{u_t}{L^2}^{\f{2(k-2-l)}{k-2}},\quad l = 1,\dots, k-3,\\
\norm{\nabla^l u}{L^2}^2 &\lesssim \norm{\nabla^{k}u}{L^2}^{\f{2(l-2)}{k-2}}\norm{\Delta u}{L^2}^{\f{2(k-l)}{k-2}},\quad  l = 3,\dots, k-1.
\end{align*}
Estimate \eqref{est:k} and the above inequalities lead to the core estimate
 \begin{align}\label{final-estimate}
&\norm{u_t(t)}{H^{k-2}}^2 + \norm{ \nabla u(t)}{H^{k-1}}^2 + \f{\varepsilon}{2} \int_0^t \norm{\nabla u_t(s)}{H^{k-2}}^2\,ds\\ \nonumber
 &\quad \lesssim \int_0^t \left[( 1 + \norm{\nabla u}{L^{\infty}}^{2k} + \norm{u_t}{L^{\infty}}^{2k})( \norm{\nabla u}{H^{k-1}}^2 
   + \norm{u_t}{H^{k-2}}^2)  \right] \,ds\\ \nonumber
 &\qquad + (1+T)(\norm{u_1}{H^{k-2}}^2 + \norm{  \nabla u_0}{H^{k-1}}^2), \qquad t \in [0,T].
 \end{align}
 for   solutions of \eqref{Reg-Cauchy-problem} and $T<T_{\epsilon,m}$. Using Gronwall's lemma we also obtain
\begin{align}\label{blowup2}
 & \sup_{t \in [0,T]}\left(\norm{u_t(t)}{H^{k-2}}^2 + \norm{ \nabla u(t)}{H^{k-1}}^2\right)\\ \nonumber
 & \quad \leq C (1+T) \left( \norm{u_1}{H^{k-2}}^2 + \norm{  \nabla u_0}{H^{k-1}}^2 \right)\exp\left(\int_0^T ( 1 + \norm{\nabla u}{L^{\infty}}^{2k} 
     + \norm{ u_t}{L^{\infty}}^{2k})\,ds\right).
\end{align}

At least for small times we want to  remove the dependence on $u$ on the right-hand side of \eqref{final-estimate} and thus we introduce the quantity
\[\alpha(t)=   \|\nabla u(t) \|_{L^2}^2 + \|\Delta  u(t)\|_{L^2}^2 + \norm{\nabla^k u(t)}{L^2}^2 + \norm{u_t(t)}{L^2}^2+ \norm{\nabla^{k-2}u_t(t)}{L^2}^2 \]
for $t\in [0, T_{\epsilon,m})$. We observe that $\alpha(t)$ is equivalent to the square of the Sobolev norms appearing in \eqref{final-estimate}.
Since the solutions to \eqref{Reg-Cauchy-problem} are (locally) unique, our reasoning is also valid for any initial
time $t_0\in(0, T_{\epsilon,m})$. The estimates \eqref{est:k}, \eqref{est:2} and \eqref{est:1} thus imply
\[\alpha(t)- \alpha(t_0)\le C\int_{t_0}^t (1+\alpha(s)^k)\alpha(s)\,ds.\]
By the above arguments, the function $\alpha$ is differentiable a.e.\ so that 
  \begin{align}\label{diff-estimate}
   \f{d}{dt}\E(t) \leq C ( 1 + \E(t)^k)\E(t)
    \end{align}
for a.e.\ $ 0\le t_0 \le t < T_{\varepsilon,m}$.   
% We note that since the solutions to \eqref{Reg-Cauchy-problem} are (locally) unique, the above reasoning also 
% yields the inequality \eqref{final-estimate} in the version
%   \begin{align}\label{final-estimate2}
%   &\norm{u_t(t)}{H^{k-2}}^2 + \norm{  u(t)}{H^{k-1}}^2 + \f{\varepsilon}{2} \int_{t_0}^t \norm{\nabla u_t(s)}{H^{k-2}}^2 ds\\ \nonumber
%   &\quad\lesssim \int_{t_0}^t \Big[( 1 + \norm{\nabla u}{L^{\infty}}^{2k} + \norm{u_t}{L^{\infty}}^{2k})( \norm{\nabla u}{H^{k-1}}^2 
%      + \norm{u_t}{H^{k-2}}^2)  \Big] \, ds\\ \nonumber
%   &\qquad + (1+T)(\norm{u_t(t_0)}{H^{k-2}}^2 + \norm{  \nabla u(t_0)}{H^{k-1}}^2).
%   \end{align}
%  for $ 0\le t_0 \le t \le T$. We set
%  $$ \E(t) := \norm{u_t(t)}{H^{k-2}}^2 + \norm{\nabla u(t)}{H^{k-1}}^2$$
%  for $t\in[0,T]$. The above arguments imply that $\alpha$ is differentiable a.e., so that \eqref{final-estimate2} implies 
%  \begin{align}\label{diff-estimate}
%   \f{d}{dt}\E(t) \leq C ( 1 + \E(t)^k)\E(t),\qquad t \in [0, T ],
%    \end{align}
% noting that $T\le T'$ for some  $T'\in(0,\infty)$ and all  $\varepsilon \in (0,1)$, cf.\ \eqref{def:Tdelta}.
 We now proceed similarly to \cite{kenig2010cauchy}, where regularization by the (intrinsic) biharmonic energy  has been applied in order to obtain 
 the existence of local Schr\"odinger maps. 
 \begin{Lemma}\label{uniform-energy-estimates}
 Let $ \varepsilon \in (0,1)$ and take data $ u_0, u_1 : \R^n \to \R^L$ with $u_0(x) \in N$ and $u_1(x) \in T_{u_0(x)}N$ 
 for a.e.\ $x \in \R^n $ satisfying
 $$ (\nabla u_0, u_1) \in H^{k-1}(\R^n) \times H^{k-2}(\R^n) \text{ \ \ for some $ k \in \N$ with $ k > \lfloor \tfrac{n}{2} \rfloor + 2 $.}  $$ 
 Let $T_{\varepsilon,m}>0$ be the maximal existence time of the solution
 $ u^{\varepsilon} : \R^n \times [0, T_{\varepsilon,m}) \to N $ of \eqref{Reg-Cauchy-problem} with $u^\varepsilon(0)=u_0$ and $\partial_t u^\varepsilon(0)=u_1$ 
 from Proposition \ref{Corollary-existence}. 
 Then there is a time $ T_0 = T_0(\| \nabla u_0\|_{H^{k-1}}, \| u_1\|_{H^{k-2}}) > 0 $ such that $T_{\varepsilon,m}> T_0$ for all $ \varepsilon \in (0,1)$.
 \end{Lemma}
 \begin{proof}
 Let $ \varepsilon \in (0,1)$ and $t\in[0,T_{\varepsilon,m})$. We write  $u = u^{\varepsilon}$.
>From \eqref{diff-estimate} we infer
 \begin{align}\label{log-bound}
 \f{d}{dt} \log\left( \f{\E}{( 1 + \E^k)^{\f{1}{k}}}\right) = \f{\E'}{( 1 + \E^k)\E} \leq C,
 \end{align}
With $\alpha_0=\alpha(0)$ it follows 
 \begin{align*} 
 \f{\E(t)^k}{( 1 + \E(t)^k)} &\leq e^{C t k}\f{\E_0^k}{( 1 + \E_0^k)} \leq ( 1 + 4Ctk ) \f{\E_0^k}{( 1 + \E_0^k)},\\
\E(t)^k &\leq (1 + 4 C t k)\E_0^k + 4 C t k \E_0^k \E^k
\end{align*}
 for $0 \leq t  \le \f{1}{8Ck}$, and hence
\begin{align*}
\E(t)^k \leq 2(1 + 4Ctk)\E_0^k \leq 3 \E_0^k
 \end{align*}
 for  $0 \leq t \le \f{1}{8Ck}\min\{ 1, \f{1}{\E_0^k}\}=:T_0$. Since $\alpha$ and the Sobolev norms are equivalent, we infer
 \begin{align}\label{uniform-energy-inequality}
\norm{u_t(t)}{H^{k-2}}^2 + \norm{\nabla u(t)}{H^{k-1}}^2 \leq c_0( \norm{u_1}{H^{k-2}}^2 + \norm{\nabla u_0}{H^{k-1}}^2) 
 \end{align}
 for $t \in [0, \min\{T_{\varepsilon,m}, T_0\})$ and some constant $c_0=c_0(k,n)>0$.
%We recall that $u = u^{\varepsilon}$ (where $ \varepsilon \in (0,1) $ is fixed) and we denote by $ T_{\varepsilon}$ the maximal existence time of $u^{\varepsilon}$.\\[5pt]
 %We denote by $ \tilde{T}_{\varepsilon}$ the maximal existence time of $u=u^{\varepsilon}$. 
 
 We now assume by contradiction that $T_{\varepsilon,m} \le T_0$ for some (fixed) $\varepsilon\in(0,1)$.
We apply the contraction argument in the proof of Lemma \ref{Lemma-existence} for the initial time  $ t_0 \in [0, T_{\varepsilon,m})$ and data $(u(t_0), u_t(t_0))$ 
in the  fixed-point space $ \mathcal{B}_r(T) $ with radius 
 $$r^k = 3 r(t_0)^k := 3\Big( \norm{\nabla u(t_0)}{H^{k-1}} + \norm{ u_t(t_0)}{H^{k-2}}\Big)^k.$$
 Since $t_0<T_0$, estimate \eqref{uniform-energy-inequality} yields the uniform bound
 \[r(t_0)\le \sqrt{2c_0} ( \norm{u_1}{H^{k-2}}^2 + \norm{\nabla u_0}{H^{k-1}}^2)^{1/2}=:\hat{c}_0.\] 
 As a result, the time 
 \begin{align*}
 T := \f{1}{4}\min \left\{ \left(\f{\sqrt[k]{3}-1}{\sqrt[k]{3}}\right)^2 \f{\varepsilon}{\hat{C}^2(1 + 3 \hat{c}_0^k)^2}, \f{\varepsilon }{\hat{C}^2 ( 1 + 6 \hat{c}_0^k)^2} \right \}.
 \end{align*}
  is less or equal than the time $T_\delta$ for $ \mathcal{B}_r(T)$ in \eqref{def:Tdelta}.  Therefore, 
 the solution can be uniquely extended to $ [0, t_0 + T] $ in the regularity class of Proposition \ref{Corollary-existence}.
 For $t_0 > T_{\varepsilon,m}-T$ this fact contradicts the maximality  of $T_{\varepsilon,m}$, showing the result.
% However, by \eqref{uniform-energy-inequality} (note here $ t_0 \in [0, T_0)$ by assumption), we succeed to solve \eqref{Reg-Cauchy-problem} by the proof of Lemma \ref{Lemma-existence} and Proposition \ref{Corollary-existence} starting from $ u(t_0),~ u_t(t_0)$ with the existence time 
% \begin{align}
% T &= \f{1}{8} \min \left\{ \left(\f{\sqrt[k]{3}-1}{\sqrt[k]{3}}\right)^2 \f{\varepsilon}{C^2(1 + 2^{k+1}3 r_0^k)^2}, \f{\varepsilon }{C^2 ( 1 + 2^{k+1} 6 r_0^k)^2} \right \}\\\nonumber
% & \leq \f{1}{8} \min \left\{ \left(\f{\sqrt[k]{3}-1}{\sqrt[k]{3}}\right)^2 \f{\varepsilon}{C^2(1 + 3 r(t_0)^k)^2}, \f{\varepsilon }{C^2 ( 1 + 6 r(t_0)^k)^2} \right \},
% \end{align}
% Since $T> 0 $ does only depend on $(u_0,u_1)$ and hence not on the choice of $ t_0 \in [0, T_{\varepsilon})$, we infer a contradiction for  $T_{\varepsilon} - t_0 < T$. 
% Thus we set $ T(u_0, u_1) = T_0$.
 \end{proof}

\section{Proof of the main theorem}\label{main-theorem-proof}
We now combine the existence result from Proposition \ref{Corollary-existence} with Lemma~\ref{uniform-energy-estimates}.
Thus, there exists a solution $ u^{\varepsilon } : \R^n \times [0, T_0] \to N $ of \eqref{Reg-Cauchy-problem} for each $\varepsilon \in (0,1)$, 
where $ T_0 > 0 $ only depends on $ \norm{\nabla u_0}{H^{k-1}}$ and $\norm{ u_1}{H^{k-2}}$. From \eqref{uniform-energy-inequality} and the inequality
$$ \norm{u^{\varepsilon} - u_0}{L^{\infty}_tL^2_x} \leq T_0 \norm{u^{\varepsilon}_t}{L_t^{\infty }L^2_x},$$
we extract a limit $ u : \R^n \times [0, T_0] \to \R^L$ as $ \varepsilon \to 0^+$
of the solutions $ u^{\varepsilon}_{|_{[0, T_0]}}$ in the sense 
\begin{align*}
&\nabla^{l_1} u^{\varepsilon} \overset{*}{\rightharpoonup}\nabla^{l_1} u,~~~u^{\varepsilon} - u_0 \overset{*}{\rightharpoonup} u - u_0,
~~\text{ and }~~\nabla^{l_2-2} u_t^{\varepsilon} \overset{*}{\rightharpoonup}  \nabla^{l_2-2} u_t~~\text{ in}~~ L^{\infty}(0,T_0; L^2),
\end{align*}
where $ 1 \leq l_1 \leq k $ and $ 0 \leq l_2 \leq k $. (Here and below we do not indicate that we pass to subsequences.) In particular,
$$ u - u_0 \in L^{\infty}(0,T_0; H^k) \cap W^{1, \infty}(0,T_0;H^{k-2}) $$
and $(\nabla u, \partial_t u) $ is weakly continuous in $ H^{k-1} \times H^{k-2}$. We first assume $ k \geq 4$ (which is no restriction if $n\ge 2$).
%and thus by differentiating \eqref{Reg-Cauchy-problem} up to order $ k - 4$ and 
Estimating the nonlinearity similarly to Section \ref{uniform-energy}, we also deduce from \eqref{eq:Nep} and \eqref{uniform-energy-inequality} 
that $ \partial_t^2 u^{\varepsilon } \in C^0([0,T_0], H^{k-4})$ is uniformly bounded as $ \varepsilon \to 0^+$. Compactness and Sobolev's embedding 
further yield
\begin{align}\label{eq:conv}
& \nabla^3 u^{\varepsilon} \to \nabla^3 u~~\text{ in}~~C^0([0,T_0],L^2_{loc}(\R^n)),\notag\\
& \partial_tu^{\varepsilon} \to \partial_t u,~~u^{\varepsilon} \to u,~~ \nabla u^{\varepsilon} \to \nabla u,~~ \nabla^2 u^{\varepsilon} \to \nabla^2 u
  ~~\text{locally~uniformly~on}~\R^n\!\times [0, T_0].
\end{align}
 More precisely for $ \alpha \in (0,1) $ and $ v^{\varepsilon} = u^{\varepsilon} - u_0$, our a priori estimates and  \cite[Prop. 1.1.4]{lunardisemigroup}
 imply uniform bounds (in $ \varepsilon$) in the spaces
\begin{align}
v^{\varepsilon} \in C^{\alpha} H^{k - 2 \alpha},~~\nabla  v^{\varepsilon} \in C^{\alpha}H^{k-1 - 2 \alpha},~~ 
  \nabla^2  v^{\varepsilon}\in C^{\alpha}H^{k -2 - 2 \alpha},~~\partial_t v^{\varepsilon} \in C^{\alpha}H^{k -2 - 2 \alpha}.   
\end{align}
As a result, $u$ takes values in $N$. Moreover, since \eqref{final-estimate} and \eqref{uniform-energy-inequality} give
\begin{align}
 &  \int_0^{T} \norm{\sqrt{\varepsilon}\nabla u^{\varepsilon}_t(s)}{H^{k-2}}^2\,ds\\[5pt]\nonumber
 &\quad \lesssim \big( T_0  (1 + \norm{u_1}{H^{k-2}}^{2k} + \norm{  \nabla u_0}{H^{k-1}}^{2k}) + 1\big) (\norm{u_1}{H^{k-2}}^2 + \norm{  \nabla u_0}{H^{k-1}}^2)
\end{align}
and $k\ge 3$, we infer that $ \varepsilon \Delta \partial_t u^{\varepsilon}  \to 0 $ in $L^2_{t,x} $. Combining this fact with \eqref{eq:conv}
and recalling \eqref{Reg-expansion}, we conclude
%Also the coefficients in \eqref{expansion} converge (locally uniformly) and from the limits above, we see (considering the definition of $\mathcal{N}_{\varepsilon}$)
\begin{align*}
\mathcal{N}_{\varepsilon}(u^{\varepsilon}) \to \mathcal{N}(u) \quad \text{ in }~ L^2_{loc}(\R^n \times [0, T_0]).
\end{align*}
%Here we note in particular that $ (I-P_{u^{\varepsilon}}) ( \Delta u^{\varepsilon}_t) $ converges in $ L^2_{loc}(\R^n \times [0,T])$ as $ \varepsilon \to 0^+$. 

In the case $ n = 1$ and $ k = 3 $ we obtain the convergence %$ \varepsilon (I-P_{u^{\varepsilon}}) ( \Delta u^{\varepsilon}_t) \to 0 $ and 
$ \mathcal{N}_{\varepsilon}(u^{\varepsilon})\to \mathcal{N}(u)$ in the sense of the duality $(H^1, H^{-1}) $ because we  still have
$$ \nabla u^{\varepsilon} \to \nabla u,\quad  \nabla^2 u^{\varepsilon} \to \nabla^2 u, \quad \partial_t u^{\varepsilon} \to \partial_t u$$
locally uniformly, as well as $ \nabla^3 u^{\varepsilon} \to \nabla^3 u$ and $ \nabla\partial_t u^{\varepsilon} \to  \nabla\partial_t u$ in $C^0([0,T_0], H^{-1}_{loc})$ 
as $ \varepsilon \to 0^+$.

Summing up, we have constructed a local solution $u:[0, T_0]\times \R^n\to N$ of \eqref{expansion} with $u(0)=u_0$ and $\partial_t u(0)=u_1$
such that $(\nabla u, \partial_t u) $ is bounded and weakly continuous in $ H^{k-1} \times H^{k-2}$.

\smallskip

In  Lemma \ref{uniqueness} it will be shown that such a solution is locally unique. We recall from the proof of Proposition \ref{uniform-energy-estimates} 
that the solution $u : \R^n \times [0, T) \to N $ for some $ T > 0 $ can be extended  if 
$ \limsup_{t \to T^-}( \norm{\nabla u(t)}{H^{k-1}} + \norm{u_t(t)}{H^{k-2}} ) < \infty $. There thus exists a maximal time of existence 
$ T_m \in (T_0, \infty] $  of $u$ with
$$\limsup_{t \to T_m^-}( \norm{\nabla u(t)}{H^{k-1}} + \norm{u_t(t)}{H^{k-2}} ) = \infty \qquad \text{if \ } T_m< \infty.$$

Arguing as in Section \ref{uniform-energy}, we establish the energy equality
\begin{align}\label{identity-final}
\norm{\nabla^{k} u}{L^2}^2 + \norm{\nabla^{k-2}u_t}{L^2}^2 &=  2 \int_0^t\int_{\R^n} \nabla^{k-2}( \mathcal{N}(u))\cdot \nabla^{k-2}u_t\,dx\,ds\\ \nonumber
&~~~  + \norm{\nabla^{k} u_0}{L^2}^2 + \norm{\nabla^{k-2} u_1}{L^2}^2
\end{align} 
for $t\in [0,T_m)$. (The integral is well-defined in view of the cancellation of one derivative in \eqref{eq:cancel}.) However, in contrast to the approximations 
$u^\varepsilon $, the solution $u$ has only $k$ weak spatial derivatives (and $\partial_t u$ has $k-2)$. For this reason, when deriving  \eqref{identity-final} 
we have to replace one spatial derivative by a difference quotient. The details are  outlined in Appendix \ref{appendix-identity-final}.

We conclude that the highest derivatives $\nabla^{k-2} u_t, \nabla^{k}u:[0, T_m)\to L^2$ are continuous, employing their weak continuity and that 
the right-hand side of \eqref{identity-final} is continuous in $t$. The continuity of the lower order derivatives can be shown as in the next section, so that 
% and  the highest derivatives $\nabla^{k} u_t$ and $\nabla^{k-2}u$ 
% Since for the lower order derivatives we proceed as in Section \ref{uniform-energy} this proves continuity of $ \norm{ (\nabla u, \partial_t u)}{H^{k-1}\times H^{k-2}} $.
% Further we already know from Section \ref{main-theorem-proof} that $   ( \nabla u, \partial_t u) $ is weakly continuous in the space $H^{k-1} \times H^{k-2} $ and thus this implies
% $$ (\nabla u, \partial_t u) \in C^0([0,T], H^{k-1}(\R^n)) \cap C^1([0,T], H^{k-2}(\R^n)).$$
% 
% is only only belongs tothe Section \ref{uniform-energy}
% Due to the loss of regularity in the $ \varepsilon$-limit we technically have to establish \eqref{identity-final} by the use of difference quotients. This is outlined in the Appendix %\ref{appendix-identity-final}.
%Thus in particular
$$  u - u_0 \in C^0([0,T_m), H^k)\cap C^1([0,T_m), H^{k-2})$$
as asserted.
Finally, following the proof of the a priori estimate in Section \ref{uniform-energy} we can derive the blow-up criterion  \eqref{blowup}, 
cf.\ Appendix \ref{appendix-identity-final}.
% Due to the loss of regularity in the $ \varepsilon$-limit we technically have to establish \eqref{identity-final} by the use of difference quotients. This is outlined in the Appendix \ref{appendix-identity-final}.

To show Theorem \ref{maintheorem} it thus remains to establish the uniqueness statement and the continuous dependence on the initial data, which is done in the 
next Sections \ref{uni} and \ref{conti}.

\section{Uniqueness}\label{uni}
\begin{Lemma}\label{uniqueness}
Let $ u, v : \R^n \times [0, T] \to N$ be two solutions of \eqref{EL-condition} with initial data $ u_0 :  \R^n \to N$ and  $u_1 : \R^n \to \R^L  $ 
such that $ u_1 \in T_{u_0} N $ on $ \R^n $ and 
$$ (\nabla u_0, u_1) \in H^{k-1} (\R^n) \times H^{k-2}(\R^n)$$
for some $ k \in \N$ with $ k > \lfloor \f{n}{2} \rfloor +2 $. Also let
$$ u - u_0,~ v- u_0  \in L^{\infty}(0, T; H^{k}(\R^n))\cap W^{1, \infty}(0, T; H^{k-2}(\R^n)). $$
Then $ u_{\vert_{ [0,T]}} = v_{\vert_{ [0,T]}}.$
\end{Lemma}
\begin{proof}[Proof of Lemma \ref{uniqueness}]
We derive the uniqueness statement from a Gronwall argument based on the equality
\begin{align}\label{derivative}
   \f{d}{dt}\f{1}{2}\int_{\R^n} |\nabla^{l} w_t|^2 + | \nabla^{l+2} w|^2\,dx = \int_{\R^n} \nabla^{l}( \mathcal{N}(u) - \mathcal{N}(v))\cdot \nabla^{l} w_t \,dx,
\end{align}
for $w=u-v$, $l \in \{ 0, \dots, k-3\} $ and $t\in[0,T]$, which is a consequence of \eqref{expansion}. Setting
$$ \mathcal{E}(t) = \norm{w(t)}{H^{k-1}}^2 + \norm{w_t(t)}{H^{k-3}}^2,$$
we want to prove
\begin{align}
	\f{d}{dt}\mathcal{E}(t) \leq C ( 1 + \norm{\nabla u}{H^{k-1}}^{2k} + \norm{u_t}{H^{k-2}}^{2k} + \norm{\nabla v}{H^{k-1}}^{2k} + \norm{v_t}{H^{k-2}}^{2k}) \mathcal{E}(t)
\end{align}
for $t \in [0, T]$. We first estimate \eqref{derivative} in the case $ l = k-3$.
Since $ u$ and  $v $ map into $N$, we have $ \mathcal{N}(u) = (I-P_u)(\mathcal{N}(u))$ and analogously for $v$. It follows
\begin{align*}
	\mathcal{N}(u) - \mathcal{N}(v) &= ( I - P_u)\mathcal{N}(u) - (I- P_v)\mathcal{N}(v)\\
	& = (P_v - P_u)  \mathcal{N}(u) + (I-P_v)(\mathcal{N}(u) - \mathcal{N}(v)),
\end{align*}
and hence
\begin{align*}
	\nabla^{k-3}(\mathcal{N}(u) - \mathcal{N}(v))\cdot \nabla^{k-3}w_t &= \nabla^{k-3} [(P_v - P_u)  \mathcal{N}(u)] \cdot \nabla^{k-3} w_t\\
	&~~~ + \nabla^{k-3}[(I-P_v)(\mathcal{N}(u) - \mathcal{N}(v))] \cdot \nabla^{k-3} w_t.
\end{align*}
In this way, we can avoid that all derivatives fall on $ \nabla^3 w $. We next write
\begin{align*}
	\nabla^{k-3} [(P_v - P_u)  \mathcal{N}(u)] \cdot \nabla^{k-3} w_t &= (P_v - P_u) \nabla^{k-3} [ \mathcal{N}(u)] \cdot \nabla^{k-3}w_t\\
	&\quad + \!\!\!\sum_{\substack{l_1 + l_2 = k-3 \\ l_1 > 0} }\!\!\nabla^{l_1}[(P_v - P_u)] \star \nabla^{l_2} [ \mathcal{N}(u)] \cdot \nabla^{k-3}w_t = : I_1 + I_2.
\end{align*}
Observe that
$$ \int_{\R^n} I_1 \,dx \lesssim \norm{w}{L^{\infty}} \norm{\nabla^{k-3} \mathcal{N}(u)}{L^2} \norm{\nabla^{k-3}w_t}{L^2}.$$
We then control $ \norm{\nabla^{k-3} \mathcal{N}(u)}{L^2} $ using Lemma \ref{Moser-estimate} as above for the a priori estimate \eqref{final-estimate}.
Further, Lemma \ref{uniqueness-lemma1}  implies that $ \int_{\R^n} I_2\,dx $ is bounded by terms of the form
\begin{align}
	&\norm{w}{L^{\infty}} \norm{|\nabla^{m_1 + 1}u| \cdots |\nabla^{m_j + 1}u | | \nabla^{l_2} \mathcal{N}(u)| }{L^2} \norm{\nabla^{k-3}w_t}{L^2}, \label{I}\\ 
	 &\norm{\nabla^{k-3}w_t}{L^2}\norm{|\nabla^{m_1 +1}w| | \nabla^{m_2 +1}h_1| \cdots |\nabla^{m_j + 1} h_{j-1}||\nabla^{l_2}\mathcal{N}(u)|}{L^2},\label{II}
\end{align}
where $m_1,\dots, m_j$ and $h_1,\dots, h_{j-1} $ are as in Lemma \ref{uniqueness-lemma1}. In \eqref{I} we then estimate as above in the a priori estimate. 
For \eqref{II}, it suffices to control terms of the form
\begin{align}\label{III}
	|\nabla^{m_1 +1}w| | \nabla^{m_2 +1}h_1| \cdots |\nabla^{m_j + 1} h_{j-1}|| \nabla^{\tilde{m}_1 + 1} u| \cdots  
	    |\nabla^{\tilde{m}_i + 1} u| \left[ | \nabla^{k_1} u_t | | \nabla^{k_2}u_t| \cdots \right ],
\end{align}
where $ \left[ | \nabla^{k_1} u_t | | \nabla^{k_2}u_t| \cdots \right ] $ is given as in the nonlinearity $ \mathcal{N}(u)$ and the orders 
$ m_1 \dots, m_j$, $\tilde{m}_1, \dots,\tilde{m}_i,$ and  $k_1, k_2 \dots $  are as used before. To apply Lemma \ref{Moser-estimate}, as above we choose  
$$ f_1 = w,~ f_2 = \nabla h_1,~ \dots, f_{j} = \nabla h_{j-1},f_{ j + 1} = \nabla u,\dots, f_{i+j} = \nabla u, $$
and $ f_{i+j+1},$ $f_{i+j+2}, \dots$, 
%($f_{i+j+3},~f_{i+j+4}$), 
according to the respective terms in $\mathcal{N}(u)$. 
We can thus estimate \eqref{III} in $L^2$ by 
\begin{align*}
&\left(\norm{w}{L^{\infty}}^{1 - \f{m_1}{k - 2 -i-j}} \norm{w}{H^{k-2-i-j}}^{\f{m_1}{k-2 - i - j}} 
    + \norm{w}{L^{\infty}}^{1 - \f{m_1}{k - 1 -i-j}} \norm{w}{H^{k-1-i-j}}^{\f{m_1}{k-1 - i - j}}\right)\\
&\qquad \cdot(1 + \norm{\nabla u}{H^{k-1}}^{2k} + \norm{u_t}{H^{k-2}}^{2k} + \norm{\nabla v}{H^{k-1}}^{2k} + \norm{v_t}{H^{k-2}}^{2k})\\
&\ \ \lesssim \norm{w}{H^{k-1}}(1 + \norm{\nabla u}{H^{k-1}}^{2k} + \norm{u_t}{H^{k-2}}^{2k} + \norm{\nabla v}{H^{k-1}}^{2k} + \norm{v_t}{H^{k-2}}^{2k}),
\end{align*}
noting that $l_1 > 0,$ $j \geq 1 $ and $ i + j < k-2$. We  continue by computing
\begin{align*}
&\nabla^{k-3}[(I-P_v)(\mathcal{N}(u) - \mathcal{N}(v))] \cdot \nabla^{k-3} w_t\\[4pt]
& = \nabla^{k-3}(\mathcal{N}(u)\! -\! \mathcal{N}(v)) (I- P_v)\nabla^{k-3}w_t
   +\!\! \sum_{\substack{l_1 + l_2 = k-3 \\ l_1 > 0} } \!\nabla^{l_1}(I - P_v)\star \nabla^{l_2}(\mathcal{N}(u)\! -\! \mathcal{N}(v)) \cdot \nabla^{k-3}w_t \\
&= \nabla^{k-3}(\mathcal{N}(u) \!- \!\mathcal{N}(v))\nabla^{k-3}[(P_u \!-\! P_v) u_t] 
- \!\!\!\!  \sum_{\substack{l_1 + l_2 = k-3 \\ l_1 > 0} }\!\!\! \nabla^{k-3}(\mathcal{N}(u)\! - \!\mathcal{N}(v))\! \cdot\! \nabla^{l_1}[(I-P_v)] \!\star\! \nabla^{l_2}w_t\\
&\quad + \sum_{\substack{l_1 + l_2 = k-3 \\ l_1 > 0} } \nabla^{l_1}(I - P_v)\star \nabla^{l_2}(\mathcal{N}(u) - \mathcal{N}(v)) \cdot \nabla^{k-3}w_t =: J_1 + J_2 + J_3.
\end{align*}
where the second equality is a consequence of
\begin{align*}
	(I-P_v) w_t = (I-P_v)u_t = [(I-P_v) - (I-P_u)]u_t = (P_u - P_v )u_t.
\end{align*}
We use integration by parts to treat $ \int J_1\,dx$ and $\int J_2 \,dx$. Here we assume that $ k \geq 4$. (If $k = 3$ the estimate becomes easier 
and we only employ integration by parts for  $ dP_v( \nabla^3 w \star \nabla u )$ in the difference $ \mathcal{N}(u) - \mathcal{N}(v)$.) It follows
\begin{align*}
\int_{\R^n} J_1 \,dx &= - \int_{\R^n} \nabla^{k-4}[\mathcal{N}(u) - \mathcal{N}(v)] \cdot \nabla^{k-2 }[(P_u - P_v)u_t]\,dx,\\
\int_{\R^n} J_2 \,dx &= \sum_{\substack{l_1 + l_2 = k-3\\l_1>0}}\int_{\R^n} \nabla^{k-4}[\mathcal{N}(u)-\mathcal{N}(v)] \cdot [\nabla^{l_1+1 }(I-P_v)\star\nabla^{l_2}w_t\\
   &\qquad + \nabla^{l_1 }(I-P_v)\star \nabla^{l_2+1}w_t  ]\,dx.
    \end{align*}
We first bound
\begin{align*}
  \int_{\R^n} J_1 \,dx \lesssim \norm{\nabla^{k-4}[\mathcal{N}(u) - \mathcal{N}(v)] }{L^2}\norm{\nabla^{k-2 }[(P_u - P_v)u_t]}{L^2}.
\end{align*}
Corollary \ref{uniqueness-corollary1}, Lemma \ref{uniqueness-lemma1} and Lemma \ref{Moser-estimate} yield
\begin{align*}
    \norm{\nabla^{k-4}[\mathcal{N}(u) - \mathcal{N}(v)] }{L^2} &\lesssim (\norm{ w }{H^{k-1}} + \norm{w_t}{H^{k-3}})\\
      &\qquad \cdot(1 + \norm{\nabla u}{H^{k-1}}^{2k} + \norm{u_t}{H^{k-2}}^{2k} + \norm{\nabla v}{H^{k-1}}^{2k} + \norm{v_t}{H^{k-2}}^{2k}),\\
\norm{\nabla^{k-2 }[(P_u - P_v)u_t]}{L^2} &\lesssim \norm{w}{H^{k-1}}(1 + \norm{\nabla u}{H^{k-1}}^{2k} + \norm{u_t}{H^{k-2}}^{2k} 
+ \norm{\nabla v}{H^{k-1}}^{2k} + \norm{v_t}{H^{k-2}}^{2k}).
\end{align*}
%     Similarly
%     \begin{align*}
%     \int_{\R^n} J_2 \,dx \lessim  \sum_{\substack{l_1 + l_2 = k-3\\l_1 > 0}} \norm{ \nabla^{k-4}[\mathcal{N}(u) - \mathcal{N}(v)]}{L^2}(  \norm{\nabla^{l_1+1 }(I-P_v)\nabla^{l_2}w_t}{L^2}  + \norm{\nabla^{l_1 }(I-P_v)\nabla^{l_2+1}w_t}{L^2}),
%     \end{align*}
%     for which we obtain similar upper bounds. For $J_3$, we note that
%     \begin{align*}
%     \int_{\R^n} J_3 \,dx \leq  \sum_{\substack{l_1 + l_2 = k-3\\l_1 > 0}} \norm{\nabla^{l_1}(I-P_v) \nabla^{l_2}[\mathcal{N}(u) - \mathcal{N}(v)]}{L^2} \norm{\nabla^{k-3 }w_t}{L^2},
%     \end{align*}
%     and (note here $ l_2 < k-3 $) again by Corollary \ref{uniqueness-corollary1} 
%     \begin{align*}
%     &\norm{\nabla^{l_1}(I-P_v) \nabla^{l_2}[\mathcal{N}(u) - \mathcal{N}(v)]}{L^2}\\
%     &~~~\lesssim (\norm{ w }{H^{k-1}} + \norm{w_t}{H^{k-3}})(1 + \norm{\nabla u}{H^{k-1}}^{2k} + \norm{u_t}{H^{k-2}}^{2k} + \norm{\nabla v}{H^{k-1}}^{2k} + \norm{v_t}{H^{k-2}}^{2k}).
%     \end{align*}
 The integrals of $J_2$ and $J_3$ are treated similary.   Summing up, we obtain
    \begin{align*}
    \f{d}{dt} \int_{\R^n}\!\!\! | \nabla^{k-3}w_t|^2 \!+ \! |\nabla^{k-1} w|^2\,dx\lesssim \mathcal{E}(t)(1 + \norm{\nabla u}{H^{k-1}}^{2k} + \norm{u_t}{H^{k-2}}^{2k} + \norm{\nabla v}{H^{k-1}}^{2k} + \norm{v_t}{H^{k-2}}^{2k}).
    \end{align*}
We can similarly  derive the  estimate (integrating $ dP_v(\nabla^3w \star \nabla u)$ by parts) 
    \begin{align*}
    \f{d}{dt} \int_{\R^n} | w_t|^2 + |\Delta w|^2\,dx \lesssim \mathcal{E}(t)(1 + \norm{\nabla u}{H^{k-1}}^{2k} + \norm{u_t}{H^{k-2}}^{2k} + \norm{\nabla v}{H^{k-1}}^{2k} + \norm{v_t}{H^{k-2}}^{2k}).
    \end{align*}
Interpolation on the left-hand side then yields 
    $$ \f{d}{dt}\mathcal{E}(t) \lesssim \mathcal{E}(t)(1 + \norm{\nabla u(t)}{H^{k-1}}^{2k} + \norm{u_t(t)}{H^{k-2}}^{2k} + \norm{\nabla v(t)}{H^{k-1}}^{2k} 
    + \norm{v_t(t)}{H^{k-2}}^{2k}).$$
By assumption,  we have $ \mathcal{E}(0) = 0$ and
    $$ \sup_{t \in [0, T]}( \norm{\nabla u(t)}{H^{k-1}}^{2k} + \norm{u_t(t)}{H^{k-2}}^{2k} + \norm{\nabla v(t)}{H^{k-1}}^{2k} + \norm{v_t(t)}{H^{k-2}}^{2k} ) < \infty,$$
so that $ \mathcal{E} = 0$ on $[0,T]$ as asserted.
\end{proof}

\section{Continuity of the flow map}\label{conti}
We now prove that the solutions of the Cauchy problem for \eqref{EL-condition} depend continuously on the initial data. As seen in the previous section, 
the difference $u - v $ of two solutions $u$ and $v$ satisfies estimates in which one loses a derivative
compared the a priori bounds such as \eqref{final-estimate} for the solutions $u$ and $v$ themselves. To deal with this problem, we apply the Bona--Smith argument, which is 
outlined e.g.\ in \cite{Tzvetkov} (for the Burgers equation) and in \cite{erdogan_tzirakis} (for the KdV equation); see also the references therein.

Let $T_m$ be the maximal existence time of the solution $u$  with initial data $(u_0,u_1)$ from Theorem \ref{maintheorem}. Fix $T_0\in(0,T_m)$. 
Take data  $(v_0,v_1)$ as in the theorem satisfying
\begin{equation}\label{est:data}
\|(u_0,u_1)- (v_0,v_1)\|_{H^k\times H^{k-2}}\le R
\end{equation}
for some $R>0$.  (We note that we have to assume $ u_0 - v_0 \in L^2 $ in order to establish the a priori estimate for 
the difference of the solutions as in the Section \ref{uni}.) 
We use regularized  data $(u_0^{\delta}, u_1^{\delta})$ and $ (v_0^{\delta}, v_1^{\delta})$ in the sense of Lemma \ref{approximation-lemma} 
from Appendix \ref{section-appendix2}, where $\delta\in(0,\delta^*]$ for some $\delta^*>0$ depending on $N$.
The corresponding solutions are denoted by $u^\delta$ and $v^\delta$. They satisfy the regularity assertions    
of part a) of Theorem \ref{maintheorem} for \emph{all} $k>\lfloor \tfrac{n}{2}\rfloor +2$.  It is crucial that the a priori estimates for $u^{\delta}$ and $v^{\delta}$ 
are uniform in $\delta$. We split $u-v$ into
$$ u - v = u - u^{\delta} + u^{\delta} - v^{\delta} + v^{\delta} - v$$
and bound each of the differences in $H^k \times H^{k-2}$.

In order to estimate  $ u^{\delta} - u $ and $v^{\delta} - v$, we use the geometric structure (as before in Section  \ref{uni}). It allows us to fix
a (small) parameter $\delta>0$ for which the differences are small in  $H^k \times H^{k-2}$. This can be done uniformly for $(v_0,v_1)$ in a certain
ball around $(u_0,u_1)$. For fixed $\delta$, one can then estimate  $u^{\delta} - v^{\delta}$ employing their extra regularity,
but paying the price of a large constant (arising from the small parameter $\delta$). We can control this constant, however, by choosing 
a small radius $R>0$ in \eqref{est:data}.

 We start with some preparations concerning the cancellations caused by the geometric constraints. As in Section  \ref{uni}, we have
\begin{align}
\mathcal{N}(u^{\delta}) - \mathcal{N}(u) &= (P_{u} - P_{u^\delta})( \mathcal{N}(u^{\delta})) + (I- P_u)( \mathcal{N}(u^{\delta}) - \mathcal{N}(u)),\notag\\
(I-P_u)( u^{\delta} - u)_t &= (P_{u^{\delta}}  - P_u)u^{\delta}_t. \label{geometric-structure}
\end{align}
We then calculate (again similar to Section  \ref{uni})
\begin{align} \label{identity}
\int_{\R^n}&\nabla^{k-2}(\mathcal{N}(u^{\delta}) - \mathcal{N}(u)) \cdot \nabla^{k-2} (u^{\delta} - u)_t\,dx \\ \nonumber
&=  \int_{\R^n}(P_{u}-P_{u^\delta}) \nabla^{k-2}[\mathcal{N}(u^{\delta})]\cdot \nabla^{k-2}(u^{\delta}-u)_t\,dx\\ \nonumber
&\quad+ \sum_{\substack{l_1+ l_2 = k-2 \\ l_1 > 0}} \int_{\R^n}\nabla^{l_1}[(P_{u^{\delta}}-P_u)]\star  \nabla^{l_2}\mathcal{N}(u^{\delta})
   \cdot  \nabla^{k-2}(u^{\delta}-u)_t\,dx\\ \nonumber
&\quad+ \sum_{\substack{l_1+ l_2 = k-2 \\ l_1 > 0}} \int_{\R^n}\nabla^{l_1}(I-P_u) \star\nabla^{l_2}[\mathcal{N}(u^{\delta}) - \mathcal{N}(u)]
  \cdot \nabla^{k-2}(u^{\delta}-u)_t\,dx\\ \nonumber
&\quad+ \int_{\R^n} \nabla^{k-2}[\mathcal{N}(u^{\delta}) - \mathcal{N}(u)]\cdot (I-P_u)\nabla^{k-2}(u^{\delta}-u)_t\,dx. 
\end{align}
Using integration by parts and \eqref{geometric-structure}, the last term is rewritten as
\begin{align}\label{identity2}
\int_{\R^n}& \nabla^{k-2}[\mathcal{N}(u^{\delta}) - \mathcal{N}(u)]\cdot (I-P_u)\nabla^{k-2}(u^{\delta}-u)_t\,dx \\
&= \sum_{\substack{l_1+ l_2 = k-2 \\ l_1 > 0}} \int_{\R^n}\nabla^{k-3}[\mathcal{N}(u^{\delta}) - \mathcal{N}(u)]\cdot 
    \nabla (\nabla^{l_1}(I - P_u) \star\nabla^{l_2}(u^{\delta}-u)_t )\,dx\notag\\
&\quad-\sum_{\substack{l_1+ l_2 = k-1 \\ l_1 > 0}}\int_{\R^n} \nabla^{k-3}[\mathcal{N}(u^{\delta}) - \mathcal{N}(u)]\cdot 
      \nabla^{l_1} [(P_{u^{\delta}}  - P_u)]\star \nabla^{l_2} u^{\delta}_t\,dx\notag\\
&\quad- \int_{\R^n} \nabla^{k-3}[\mathcal{N}(u^{\delta}) - \mathcal{N}(u)]\cdot  (P_{u^{\delta}}  - P_u) \nabla^{k-1}u^{\delta}_t\,dx,  \notag
\end{align}
which is well defined by the higher regularity of $ u^{\delta}$. Technically this has to be established by difference quotients as 
in Appendix \ref{appendix-identity-final}, however we omit the details here. The advantage of estimating $u^{\delta} - u$ is that the \textit{bad~terms} 
(with respect to the regularity of $u$) 
\begin{equation}\label{eq:bad}
\norm{\nabla^{k-2}\mathcal{N}(u^{\delta})}{L^2} \quad \text{ and } \quad \norm{\nabla^{k-1}u^{\delta}_t}{L^2}
\end{equation}
will be bounded by the regularized initial data from Lemma \ref{approximation-lemma}. Their norm will grow  as $ \delta \to  0^+$ in a controlled way.
Moreover, when estimating \eqref{identity} and \eqref{identity2}, these bad terms only appear in the products
\begin{align*}
&\norm{u^{\delta}- u}{L^{\infty}}\norm{\nabla^{k-2}\mathcal{N}(u^{\delta})}{L^2}\norm{\nabla^{k-2}(u^{\delta}- u)_t}{L^2},\\
&\norm{u^{\delta}- u}{L^{\infty}}\norm{\nabla^{k-3}(\mathcal{N}(u^{\delta}) - \mathcal{N}(u))}{L^2}\norm{\nabla^{k-1}u^{\delta}_t}{L^2}.
\end{align*}
Here the decay of $\norm{u^{\delta}- u}{L^{\infty}}$ as $ \delta \to 0^+$ will compensate the growth in \eqref{eq:bad}. We now carry out the details
in several steps.

\smallskip

\textit{Step~1.} Since $T_0<T_m$, we have the bound
$$ \sup_{t \in [0, T_0]}\left( \norm{\nabla u(t)}{H^{k-1}} + \norm{u_t(t)}{H^{k-2}}\right) =: \overline{C} < \infty.$$
Lemma \ref{approximation-lemma} allows us to fix a parameter $\delta_1'\in(0,\delta^*]$ depending on $(u_0,u_1)$ such that  
\begin{equation} \label{est0}
\norm{(\nabla u_0^\delta, u_1^\delta)}{H^{k-1}\times H^{k-2}}\le 3 \overline{C}/2
\end{equation}
for all $\delta\in (0, \delta_1']$. We let $\delta\in (0, \delta_1']$ and  also $R\le \overline{C}/2$ in \eqref{est:data}. Hence
\begin{align}\label{est1}
\norm{(\nabla v_0, v_1)}{H^{k-1}\times H^{k-2}} &\leq \norm{(\nabla u_0, u_1)}{H^{k-1}\times H^{k-2}}  + R \leq  3\overline{C}/2,\\ \label{est2}
\norm{(\nabla v_0^\delta, v_1^\delta)}{H^{k-1}\times H^{k-2}} 
&\le \norm{(\nabla u_0^\delta, u_1^\delta)}{H^{k-1}\times H^{k-2}} + 2C_0R \le 2\overline{C}.
\end{align}
Here the constant $C_0\ge1$ is given by \eqref{fünfapp} and we have chosen $0<R<\min\{1, \overline{C}/(4C_0)\}=:R_0$.
% if $ R > 0$ is small enough. We now choose regularized initial data $ (u_0^{\delta}, u_1^{\delta})$ and $(v_0^{\delta}, v_1^{\delta})$ as in Lemma \ref*{approximation-lemma} for $ \delta > 0 $ small enough such that
% \begin{align}\label{est2}
% \norm{(\nabla v_0^{\delta}, v_1^{\delta})}{H^{k-1}\times H^{k-2}} +  \norm{(\nabla u_0^{\delta}, u_1^{\delta})}{H^{k-1}\times H^{k-2}} \leq 3C.
% \end{align}
We define a time $\tilde{T}_0>0$ as in  Lemma \ref{uniform-energy-estimates}, replacing $\alpha(0)$ there by a multiple of $\overline{C}$.
We then combine the uniform a priori bound \eqref{uniform-energy-inequality} for the approximate solution to the $\epsilon$--problem for $v$
%$$ \norm{(\nabla v, v_t)}{L^{\infty}_{\tilde{T}_0}(H^{k-1}\times H^{k-2})}^2 \lesssim c_0\norm{(\nabla v_0, v_1)}{H^{k-1}\times H^{k-2}}^2$$
on $ [0, \tilde{T}_0]$ with \eqref{est1}. Likewise one treats $ u^{\delta}$ and $v^{\delta}$ using \eqref{est0} and \eqref{est2}, respectively. Following
the existence proof  in Section  \ref{main-theorem-proof}, we then see that the solutions 
$ u_{_{[0, \tilde{T}_0]}},$ $v_{_{[0, \tilde{T}_0]}},$  $u^{\delta}_{_{[0, \tilde{T}_0]}},$ and  $v^{\delta}_{_{[0, \tilde{T}_0]}}$
exist on $[0, \tilde{T}_0]$. Proceeding as in Sections  \ref{uniform-energy} and \ref{uni}, we further obtain a constant
$ \tilde{C} = \tilde{C}(N,k,\tilde{T}_0) > 0$ such that
\begin{align}
\norm{\nabla u}{H^{m}}^2 + \norm{u_t}{H^{m-1}}^2 &\leq \tilde{C} (\norm{\nabla u_0}{H^{m}}^2 + \norm{u_1}{H^{m-1}}^2),\label{a-priori-est1}\\
\norm{\nabla v}{H^{m}}^2 + \norm{v_t}{H^{m-1}}^2 &\leq \tilde{C}(\norm{\nabla v_0}{H^{m}}^2 + \norm{v_1}{H^{m-1}}^2) ,\label{a-priori-est2}\\
\norm{ u -v}{H^{m}}^2 + \norm{u_t -v_t}{H^{m-2}}^2 &\leq \tilde{C} (\norm{ u_0 - v_0}{H^{m}}^2 + \norm{u_1 - v_1}{H^{m-2}}^2).\label{a-priori-est-diff}
\end{align}
on $ [0, \tilde{T}_0]$ and for orders $ m \in \{2, \dots,k-1\}$.  Analogously, $ u^{\delta} $ and $ v^{\delta} $ satisfy 
the estimates \eqref{a-priori-est1} respectively \eqref{a-priori-est2}, and the differences $ u - u^{\delta}$,  $v - v^{\delta}$ and $u^{\delta} - v^{\delta} $
fufill  \eqref{a-priori-est-diff} with the same constant $ \tilde{C}$ independent of $\delta\in (0, \delta^*]$.
For the regularized data we can replace here $k$ by $k+1$, deriving 
%Further from the a priori bound \eqref{blowup} and the continuous embedding $ H^{k-1}\times H^{k-2} \hookrightarrow C_0\times C_0$ we have
\begin{align}\label{a-piori-est1-higher}
\norm{\nabla u^{\delta}}{H^{k}}^2 + \norm{u_t^{\delta}}{H^{k-1}}^2 &\leq \tilde{C} (\norm{\nabla u_0^{\delta}}{H^{k}}^2 + \norm{u_1^{\delta}}{H^{k-1}}^2),\\
\norm{\nabla v^{\delta}}{H^{k}}^2 + \norm{v_t^{\delta}}{H^{k-1}}^2 &\leq \tilde{C} (\norm{\nabla v_0^{\delta}}{H^{k}}^2 + \norm{v_1^{\delta}}{H^{k-1}}^2).\notag
\end{align}

%In fact we define $ \tilde{T}_0 > 0 $ as in Lemma \ref{uniform-energy-estimates} via the constant $C$ instead of the norms of the initial data. 
% and that there is a constant $\tilde{C}_k > 0$ such that
% \begin{align}\label{est3}
% &\norm{(\nabla v, v_t)}{L^{\infty}_{\tilde{T}_0}(H^{k-1}\times H^{k-2})} + \norm{(\nabla u, u_t)}{L^{\infty}_{\tilde{T}_0}(H^{k-1}\times H^{k-2})}\\ \nonumber
% &~+\norm{(\nabla u^{\delta}, u^{\delta}_t)}{L^{\infty}_{\tilde{T}_0}(H^{k-1}\times H^{k-2})} 
%  + \norm{(\nabla v^{\delta}, v^{\delta}_t)}{L^{\infty}_{\tilde{T}_0}(H^{k-1}\times H^{k-2})} \leq \tilde{C}_k.
% \end{align}
% For the regularized data we can replace here $k$ by $k+1$, obtaining
% Further from the a priori bound \eqref{blowup} and the continuous embedding $ H^{k-1}\times H^{k-2} \hookrightarrow C_0\times C_0$ we have
% \begin{align}\label{a-piori-est1-higher}
% &\norm{\nabla u^{\delta}}{H^{k}}^2 + \norm{u_t^{\delta}}{H^{k-1}}^2 \leq \tilde{C}_k (\norm{\nabla u_0^{\delta}}{H^{k}}^2 + \norm{u_1^{\delta}}{H^{k-1}}^2) e^{\tilde{T}_0 \tilde{C}_k},\\
% &\norm{\nabla v^{\delta}}{H^{k}}^2 + \norm{v_t^{\delta}}{H^{k-1}}^2 \leq \tilde{C}_k (\norm{\nabla v_0^{\delta}}{H^{k}}^2 + \norm{v_1^{\delta}}{H^{k-1}}^2) e^{\tilde{T}_0 \tilde{C}_k}.
% \end{align}

\smallskip

\textit{Step~2.} Estimating  \eqref{identity} and \eqref{identity2}  as in Section  \ref{uni}, we derive
\begin{align*}
\f{d}{dt}\big(\norm{  u - u^{\delta}}{H^{k}}^2 + \norm{ u_t - u^{\delta}_t}{H^{k-2}}^2 \big) 
&\leq C \norm{u - u^{\delta}}{L^{\infty}} \norm{\nabla^{k-2} \mathcal{N}(u^{\delta})}{L^2} \norm{\nabla^{k-2}(u_t - u^{\delta}_t)}{L^2}\\
&\ \ + C   \norm{ u - u^{\delta}}{L^{\infty}} \! \norm{\nabla^{k-3}(\mathcal{N}(u^{\delta})\! -\! \mathcal{N}(u))}{L^2} \! \norm{\nabla^{k-1}u^{\delta}_t}{L^2}\\
&\ \ + C (\norm{  u -  u^{\delta}}{H^{k}}^2 + \norm{u_t - u^{\delta}_t}{H^{k-2}}^2)
\end{align*}
for some $ C = C(N,\overline{C}, \tilde{C}) > 0 $.  The nonlinearities are treated as in Sections  \ref{uniform-energy} and \ref{uni}.
Using also \eqref{a-priori-est1}, \eqref{a-priori-est-diff} and \eqref{a-piori-est1-higher}, we then conclude
\begin{align*}
\f{d}{dt}&\big(\norm{  u - u^{\delta}}{H^{k}}^2 + \norm{ u_t - u^{\delta}_t}{H^{k-2}}^2 \big)\\
& \leq C\norm{u - u^{\delta}}{H^{k-1}} ( 1+ \norm{\nabla u^{\delta} }{H^{k}}+ \norm{u^{\delta}_t }{H^{k-2}})(\norm{u_t}{H^{k-2}} + \norm{u^{\delta}_t}{H^{k-2}})\\
&\quad + C\norm{u - u^{\delta}}{H^{k-1}}(1+  \|\nabla u\|_{H^{k-1}} + \|\nabla u^\delta\|_{H^{k-1}} + \|u_t\|_{H^{k-3}} + \| u_t^\delta\|_{H^{k-3}} )
  \norm{ u^{\delta}_t}{H^{k-1}} \\ &\quad + C (\norm{  u -  u^{\delta}}{H^{k}}^2 + \norm{u_t - u^{\delta}_t}{H^{k-2}}^2)\\
&\leq C(\norm{u_0 - u_0^{\delta}}{H^{k-1}} +\norm{u_t - u^{\delta}_t}{H^{k-3}}) (1+\norm{\nabla u_0^{\delta} }{H^{k}} + \norm{ u^{\delta}_1}{H^{k-1}})\\
&\quad+ C (\norm{  u -  u^{\delta}}{H^{k}}^2 + \norm{u_t - u^{\delta}_t}{H^{k-2}}^2)
\end{align*}
on $[0, \tilde{T}_0]$. Gronwall's inequality  and Lemma \ref{approximation-lemma} thus yield
\begin{align*}
&\sup_{ t \in [0, \tilde{T}_0]}\big(\norm{  u -  u^{\delta}}{H^k}^2 + \norm{ u_t - u^{\delta}_t}{H^{k-2}}^2\big) \\
&\quad\leq \f{C\tilde{T}_0}{\sqrt{\delta}}  \big(\norm{u_0 - u_0^{\delta}}{H^{k-1}}  +\norm{u_1 - u_1^{\delta}}{H^{k-3}}\big)
+ C \big(\norm{ u_0 - u^{\delta}_0}{H^{k}}^2 + \norm{u_1 - u^{\delta}_1}{H^{k-2}}^2\big) = o(1) 
\end{align*}
as $\delta \to 0^+.$ In view of our a priori bounds, we can estimate  $ v - v^{\delta} $  in the same way.
Here we have to split the initial values, obtaining
% and $ u^{\delta} - v^{\delta}$ with the help of Lemma \ref{approximation-lemma} and by estimating
% \begin{align}
% \norm{v - v^{\delta}}{}\leq \norm{v - u}{} + \norm{ u - u^{\delta} }{} +\norm{u^{\delta} - v^{\delta}}{}.
% \end{align}
\begin{align*}
\sup_{ t \in [0, \tilde{T}_0]}&\big(\norm{  v -  v^{\delta}}{H^k}^2 + \norm{ v_t - v^{\delta}_t}{H^{k-2}}^2\big)\\
&\leq \f{C\tilde{T}_0}{\sqrt{\delta}}( \norm{v_0 - v_0^{\delta}}{H^{k-1}} + \norm{v_1 - v_1^{\delta}}{H^{k-3}}) + C (\norm{ v_0 - v^{\delta}_0}{H^{k}}^2 + \norm{v_1 - v^{\delta}_1}{H^{k-2}}^2)\\
&\leq \f{C\tilde{T}_0}{\sqrt{\delta}}( \norm{u_0 - u_0^{\delta}}{H^{k-1}} + \norm{u_1-u_1^{\delta}}{H^{k-3}}) + C (\norm{u_0 - u_0^{\delta}}{H^{k}}^2 + \norm{u_1 - u_1^{\delta}}{H^{k-2}}^2 )\\
&\quad+\f{C\tilde{T}_0}{\sqrt{\delta}}(\norm{u_0 - v_0}{H^{k-1}} + \norm{u_1 - v_1}{H^{k-3}} + \norm{u_0^{\delta} - v_0^{\delta}}{H^{k-1}} + \norm{u_1^{\delta} - v_1^{\delta}}{H^{k-3}} )\\
&\quad+ C (\norm{u_0 - v_0}{H^{k}}^2 + \norm{u_1 - v_1}{H^{k-2}}^2 + \norm{u_0^{\delta} - v_0^{\delta}}{H^{k}}^2 + \norm{u_1^{\delta} - v_1^{\delta}}{H^{k-2}}^2 ).
\end{align*}
Lemma \ref{approximation-lemma} now implies that
\begin{align*}
\sup_{ t \in [0, \tilde{T}_0]}&\big(\norm{  v -  v^{\delta}}{H^k}^2 + \norm{ v_t - v^{\delta}_t}{H^{k-2}}^2\big)\\
&\leq \f{C\tilde{T}_0}{\sqrt{\delta}}( \norm{u_0 - u_0^{\delta}}{H^{k-1}} + \norm{u_1-u_1^{\delta}}{H^{k-3}}) + C (\norm{u_0 - u_0^{\delta}}{H^{k}}^2 + \norm{u_1 - u_1^{\delta}}{H^{k-2}}^2 )\\
&\quad+\f{C\tilde{T}_0}{\sqrt{\delta}} R + C R^2.
\end{align*}
On the regularized level, we use the coarse estimate
\begin{align*}
\sup_{ t \in [0, \tilde{T}_0]}\big(\norm{  u^{\delta} -  v^{\delta}}{H^k}^2 + \norm{ u_t^{\delta} - v^{\delta}_t}{H^{k-2}}^2\big) 
&\leq \f{C}{\sqrt{\delta}}\tilde{T}_0( \norm{v_0^{\delta} - u_0^{\delta}}{H^{k}} + \norm{v_1^{\delta}-u_1^{\delta}}{H^{k-2}})\\
&\quad+C(\norm{u_0^{\delta} - v_0^{\delta}}{H^{k}}^2 + \norm{u_1^{\delta} - v_1^{\delta}}{H^{k-2}}^2)\\
&\leq \f{C\tilde{T}_0}{\sqrt{\delta}}R + CR^2.
\end{align*}
Since $u - v = u - u^{\delta} + u^{\delta} - v^{\delta} + v^{\delta} - v$, it follows
\begin{align}\label{est:pre-final}
\sup_{ t \in [0, \tilde{T}_0]}\big(\norm{  u -  v}{H^k}^2 + \norm{ u_t - v_t}{H^{k-2}}^2\big) 
&\leq \f{C\tilde{T}_0}{\sqrt{\delta}}( \norm{u_0 - u_0^{\delta}}{H^{k-1}} + \norm{u_1 - u_1^{\delta}}{H^{k-3}})\\
&\quad+ C (\norm{ u_0 - u^{\delta}_0}{H^{k}}^2 + \norm{u_1 - u^{\delta}_1}{H^{k-2}}^2)\notag\\
&\quad+\f{C\tilde{T}_0}{\sqrt{\delta}}R + C  R^2.\notag
\end{align}

Now take $\eta\in(0,\overline{C}/2]$ and $r_1\in (0,\eta]$. We first fix $\delta=\delta_1=\delta_1(r_1)\in(0,\delta_1']$ and then choose 
$R_1 =  R_1 (\delta_1) \in(0,R_0]$ such that for all
$R\in(0,R_1]$ we have
\begin{equation}\label{final-est}
\sup_{ t \in [0, \tilde{T}_0]}\big(\norm{  u -  v}{H^k}^2 + \norm{ u_t - v_t}{H^{k-2}}^2\big) \le r_1\le \eta.
\end{equation}
% 
% Hence for $  R > 0 $ and   $ \delta > 0 $ (small enough) it follows as $ \delta \to 0^+$
% \begin{align}\label{final-est}
% &\sup_{ t \in [0, \tilde{T}_0]}\left(\norm{  u -  v}{H^k}^2 + \norm{ u_t - v_t}{H^{k-2}}^2\right) \leq  o(1) + C(1+\tilde{T}_0)(1 + R^k)\f{R}{\sqrt{\delta}} \leq o(1) + C(1+\tilde{T}_0)\f{R}{\sqrt{\delta}}.\\ \nonumber
% &o(1)  = C( \norm{u_0 - u_0^{\delta}}{H^{k-1}} + \norm{u_1 - u_1^{\delta}}{H^{k-3}})\f{1}{\sqrt{\delta}} + C (\norm{ u_0 - u^{\delta}_0}{H^{k}}^2 + \norm{u_1 - u^{\delta}_1}{H^{k-2}}^2).
% \end{align}

In the above reasoning we now replace $(u_0,u_1)$ with corresponding solution $u$ by data $(\hat{u}_0,\hat{u}_1)$ with  solution $\hat{u}$ that satisfy
the same assumptions as $(v_0,v_1)$. The function $\hat{u}$ thus fulfills the same a priori estimates as $v$ and also \eqref{final-est}. Moreover, we assume that
\begin{equation}\label{est:data-hat}
\|(\hat{u}_0,\hat{u}_1)- (v_0,v_1)\|_{H^k\times H^{k-2}}\le \hat{R}
\end{equation}
for some radius $\hat{R}>0$.
We can then repeat the above arguments replacing $u$ by $\hat{u}$. The resulting regularization parameter $\hat{\delta}_1$ depends on $\hat{u}$,
and thus also the upper bound
$\hat{R}_1 =  \hat{R}_1 (\delta_1)$ for the radii in \eqref{est:data-hat}. For given $0\le \hat{r}_1\le \hat{\eta}$, we  infer
\begin{equation}\label{final-est-hat}
\sup_{ t \in [0, \tilde{T}_0]}\big(\norm{  \hat{u} -  v}{H^k}^2 + \norm{ \hat{u}_t - v_t}{H^{k-2}}^2\big) \le \hat{r}_1\le \hat{\eta}
\end{equation}
provided that $0< \hat{R}\le \hat{R}_1$ in  \eqref{est:data-hat}.

\smallskip

\textit{Step~3.}
In the case $ \tilde{T}_0 \geq T_0 $ the proof is complete. Otherwise we repeat the same argument starting from 
\[ (u_{0}^{(1)}, u_{1}^{(1)})=   (u(\tilde{T}_0), u_t(\tilde{T}_0)) \qquad \text{and}\quad (v_{0}^{(1)}, v_{1}^{(1)})=   (v(\tilde{T}_0), v_t(\tilde{T}_0)).\]
Observe that  \eqref{final-est} yields
\[ \norm{(\nabla v_0^{(1)}, v_1^{(1)})}{H^{k-1} \times H^{k-2}} \leq \eta + \norm{(\nabla u_{0}^{(1)}, u_{1}^{(1)})}{H^{k-1} \times H^{k-2}} 
    \leq 3\overline{C}/2.\]
For a sufficiently small $\delta_2'\in(0,\delta^*]$ and all $\delta\in (0, \delta_2']$, we derive 
\[\norm{(\nabla  (u_0^{(1)})^{\delta}, (u_1^{(1)})^{\delta})}{H^{k-1}\times H^{k-2}}, 
  \norm{(\nabla  (v_0^{(1)})^{\delta}, (v_1^{(1)})^{\delta})}{H^{k-1}\times H^{k-2}} \leq 2\overline{C}\]
as in \eqref{est0} and \eqref{est2}.
% This follows from $ o(1) \leq \f{C}{2}$ in the rhs of \eqref{final-est} for $ \delta > 0 $ small and $ R > 0 $ such that
% \begin{align}\label{est-against-C}
% o(1) + C(1+\tilde{T}_0)\f{R}{\sqrt{\delta}} \leq \f{C}{2} + C(1+\tilde{T}_0)\f{R}{\sqrt{\delta}} \leq C. 
% \end{align}
Based on these bounds we can repeat the arguments of Steps 1 and 2 on the interval $[\tilde{T}_0,\min\{2\tilde{T}_0, T_0\}]=:J_1$. 
However we have to replace the bound \eqref{est:data} involving $R$ by \eqref{final-est} which yields
\[ \|(u_0^{(1)},u_1^{(1)})- (v_0^{(1)},v_1^{(1)})\|_{H^k\times H^{k-2}}\le r_1.\]
Let $r_2\in(0,\eta]$.  Lemma \ref{approximation-lemma} allows us to fix  a parameter $\delta=\delta_2=\delta_2(r_2)\in(0,\delta_2']$ such that 
\begin{align*}
&\f{C\tilde{T}_0}{\sqrt{\delta}}  \big(\|u_0^{(1)} - (u_0^{(1)})^{\delta}\|_{H^{k-1}}  +\|u_0^{(1)}  - (u_0^{(1)})^{\delta}\|_{H^{k-3}}\big)\\
&\qquad + C \big(\|u_0^{(1)} - (u_0^{(1)})^{\delta}\|_{H^{k}}  +\|u_0^{(1)}  - (u_0^{(1)})^{\delta}\|_{H^{k-2}}\big)  \le r_2/4.
\end{align*}
As in \eqref{est:pre-final} we then obtain
\begin{equation*}
\sup_{ t \in J_1}\big(\norm{  u -  v}{H^k}^2 + \norm{ u_t - v_t}{H^{k-2}}^2\big) 
    \le r_2/ 4 + r_2/4 + \f{C\tilde{T}_0}{\sqrt{\delta_2}}r_1 + C  r_1^2  \le r_2\le \eta
\end{equation*}
if we choose $r_1$, and hence $R$, small enough. 

Again we can argue in the same way for $\hat{u}$ instead of $u$, replacing $r_i$, $\delta_i$ and  $R$ by  $\hat{r}_i$, $\hat{\delta}_i$ and  $\hat{R}$.
For given  $0<\hat{r}_2\le \hat\eta$, we thus obtain
\begin{equation*}
\sup_{ t \in J_1}\big(\norm{  \hat{u} -  v}{H^k}^2 + \norm{ \hat{u}_t - v_t}{H^{k-2}}^2\big) 
    \le \hat{r}_2/ 4 + \hat{r}_2/4 + \f{C\tilde{T}_0}{\sqrt{\hat{\delta}_2}}\hat{r}_1 + C  \hat{r}_1^2  \le \hat{r}_2\le \hat\eta
\end{equation*}
if $\hat{r}_1$ and $\hat{R}$ are  small enough.

\smallskip

\textit{Step~4.}

The previous step  can be repeated $m$ times until $m\tilde{T}_0\ge T_0$. We set $R_0=R(\overline{C}/2)$ (with $\eta=\overline{C}/2$)
and use the resulting radius $\hat{R}=\hat{R}(\hat\eta)$ for the contunuity at $\hat{u}$,
concluding the proof of the continuous dependence and thus of Theorem~\ref{maintheorem}.

\appendix 

\section{Derivatives of the nonlinearity}\label{section-appendix}
In this section  we assume $u, v : \R^n \times[0,T] \to \R^L$  are smooth maps. The calculations hold if $ u$ and $v$ are sufficiently regular to apply the Leibniz formula (e.g. with weak derivatives in $L^2$).
Lemma \ref{Lemma-Leibniz} and the Leibniz formula imply the following substitution rule.
\begin{Lemma}\label{diff-Lemma}
	Let $ l \in \N$. Then we have
	\begin{align*}
	\nabla^l( \mathcal{N}(u))= J_1 + J_2 + J_3,
	\end{align*}
	where the terms  $J_1,$ $J_2,$ and $J_3$ are of the form (with $k_i, m_i \in \N_0$)
	\begin{align*}
	J_1 \!= \! \sum_{(*)}\!d^{j+1}\!P_u (\nabla^{m_1+1} u \!\star\! \dots\! \star\! \nabla^{m_j +  1}u )
	  [ \nabla^{k_1} u_t \!\star\! \nabla^{k_2}u_t + \nabla^{k_1 +2} u\! \star\! \nabla^{k_2 +2} u + \nabla^{k_1 +3} u \!\star\! \nabla^{k_2 +1} u ]
	\end{align*}
	with $ (*) : 0 \leq m \leq l,~ \sum_{i=1}^2 k_i = l -m,~ j = \min \{ 1,m\}, \dots, m,~\sum_{k=1}^j m_k = m -j$;
	\begin{align*}
	J_2 =  \sum_{(*)} d^{j+2}P_u (\nabla^{m_1+1} u \star \dots \star \nabla^{m_j +  1}u )
	     [  \nabla^{k_1 +1} u \star \nabla^{k_2 + 1} u \star \nabla^{k_3 + 2} u ]
	\end{align*}
	with $ (*) : 0 \leq m \leq l,~ \sum_{i=1}^3 k_i = l -m,~ j = \min \{ 1,m\}, \dots, m,~\sum_{k=1}^j m_k = m -j$;
	\begin{align*}
	J_3 =  \sum_{(*)} d^{j+3}P_u (\nabla^{m_1+1} u \star \dots \star \nabla^{m_j +  1}u )
	   [  \nabla^{k_1 +1} u \star \nabla^{k_2 + 1} u\star \nabla^{k_3 + 1} u \star \nabla^{k_4 + 1} u ]
	\end{align*}
	with $ (*) : 0 \leq m \leq l,~ \sum_{i=1}^4 k_i = l -m,~ j = \min \{ 1,m\}, \dots, m,~\sum_{k=1}^j m_k = m -j$.
\end{Lemma}

The following lemmata are used to prove the existence of a fixed point in Section \ref{ex} and the uniqueness result in Section \ref{uni}.
\begin{Lemma}\label{uniqueness-lemma1}
	Let $ m \in \N,$ $k \in \N_0,$  and $ w = u-v $. For $m \geq 2$ we have
	\begin{align}
	\nabla^m( d^kP_u - d^k P_v ) &= \sum_{j=1}^m \sum_{ m_1 + \dots + m_j = m-j} (d^{j+k}P_u -d^{j+k}P_v )( \nabla^{m_1 + 1}u, \dots, \nabla^{m_j +1}u)\\ \nonumber
	&~~ + \sum_{j=2}^m \sum_{ m_1 + \dots +  m_j = m-j} d^{j+k}P_v ( \nabla^{m_1 + 1}w, \nabla^{m_2 + 1}u,\dots, \nabla^{m_j +1}u)\\\nonumber
	&~~+ \sum_{j=2}^m \sum_{ m_1 + \dots +  m_j = m-j} d^{j+k}P_v ( \nabla^{m_1 + 1}v, \nabla^{m_2 + 1}w,\nabla^{m_3+1}u,\dots, \nabla^{m_j +1}u)\\\nonumber
	&\quad\colon~~\\\nonumber
	&~~+ \sum_{j=2}^m \sum_{ m_1 + \dots +  m_j = m-j} d^{j+k}P_v ( \nabla^{m_1 + 1}v,\dots, \nabla^{m_{j-1} +1}v,\nabla^{m_{j} +1}w ),
	\end{align}	
	and for $ m = 1$
	\begin{align}
	\nabla( d^kP_u - d^k P_v ) = (d^kP_u - d^kP_v)(\nabla u) + d^kP_v ( \nabla w).
	\end{align}	
\end{Lemma}
\begin{proof}
	The result follows from subtracting the expansion in Lemma \ref{Lemma-Leibniz} for $d^kP_v$  
	$$ \nabla^m ( d^kP_v) = \sum_{j=1}^m \sum_{ m_1 + \dots + m_j = m-j} d^{j+k}P_v(\nabla^{m_1 + 1}v \star \dots \star \nabla^{m_j +1}v),$$
	from the same expansion of $\nabla^m(d^kP_u)$. Then subsequently adding and subtracting the intermediate terms in the formula above gives the result.
\end{proof}

\begin{Corollary}\label{uniqueness-corollary1} 
Let $ m \in \N,$ $k \in \N_0,$  and $ w = u-v $. Then we have
	\begin{align*}
	\nabla^{m}&\left[( dP_u - dP_v)( u_t \cdot u_t + \nabla^2u \star \nabla^2u + \nabla^3 u \star \nabla u)\right]\\ \nonumber
	&=~\sum_{(*)} (d^{j+1}P_u -d^{j+1}P_v )( \nabla^{m_1 + 1}u, \dots, \nabla^{m_j +1}u)( \nabla^{k_1}u_t \star \nabla^{k_2} u_t\\
	&\hspace{4cm} + \nabla^{k_1+2}u \star \nabla^{k_2+2} u + \nabla^{k_1+3}u \star \nabla^{k_2} u)\\ \nonumber
	&~~+ \sum_{(**)} d^{j+1}P_v ( \nabla^{m_1 + 1}w, \nabla^{m_2 + 1}u,\dots, \nabla^{m_j +1}u)( \nabla^{k_1}u_t \star \nabla^{k_2} u_t\\
	&\hspace{4cm} + \nabla^{k_1+2}u \star \nabla^{k_2+2} u + \nabla^{k_1+3}u \star \nabla^{k_2} u)\\
	%&~~+\sum_{(**)} d^{j+1}P_v ( \nabla^{m_1 + 1}v, \nabla^{m_2 + 1}w,\dots, \nabla^{m_j +1}u)( \nabla^{k_1}u_t \star \nabla^{k_2} u_t\\
	%&\hspace{4cm} + \nabla^{k_1+2}u \star \nabla^{k_2+2} u + \nabla^{k_1+3}u \star \nabla^{k_2} u),~\\
	&~~\colon~~\\
	&~~+ \sum_{(**)} d^{j+1}P_v ( \nabla^{m_1 + 1}v,\dots, \nabla^{m_{j-1} +1}v, \nabla^{m_{j} +1}w)( \nabla^{k_1}u_t \star \nabla^{k_2} u_t\\
	&\hspace{4cm} + \nabla^{k_1+2}u \star \nabla^{k_2+2} u + \nabla^{k_1+3}u \star \nabla^{k_2} u),~
	\end{align*}
	where $ (*) : j = 1,  \dots ,m~\text{and}~ m_1 + \dots + m_j + k_1 + k_2 = m - j,$ and
	$(**) : j = 2,  \dots ,m~\text{and}~ m_1 + \dots + m_j + k_1 + k_2 = m - j$. Likewise we have
	\begin{align*}
	\nabla^{m}&\left[( d^2P_u - d^2P_v)( \nabla u \star \nabla u \star \nabla^2u)\right]\\ \nonumber
	&=~\sum_{(*)} (d^{j+2}P_u -d^{j+2}P_v )( \nabla^{m_1 + 1}u, \dots, \nabla^{m_j +1}u)( \nabla^{k_1+1}u \star \nabla^{k_2+1} u \star \nabla^{k_3+2}u)\\ \nonumber
	&~~+ \sum_{(**)} d^{j+2}P_v ( \nabla^{m_1 + 1}w, \nabla^{m_2 + 1}u,\dots, \nabla^{m_j +1}u)( \nabla^{k_1+1}u \star \nabla^{k_2+1} u \star \nabla^{k_3+2}u)\\
	&~~\colon~~\\
	&~~+ \sum_{(**)} d^{j+2}P_v ( \nabla^{m_1 + 1}v, \dots, \nabla^{m_{j-1} +1}v,\nabla^{m_{j} +1}w )( \nabla^{k_1+1}u \star \nabla^{k_2+1} u \star \nabla^{k_3+2}u)
	\end{align*}
	where $ (*) : j = 1,  \dots ,m~\text{and}~m_1 + \dots + m_j + k_1 + k_2 + k_3 = m - j,$ and $(**) : j = 2,  \dots ,m~\text{and}~ m_1 + \dots + m_j + k_1 + k_2 + k_3 = m - j$. Further
	\begin{align*}
	&\nabla^{m}\left[( d^3P_u - d^3P_v)( \nabla u \star \nabla u \star \nabla u\star \nabla u)\right]\\ \nonumber
	&=\sum_{(*)} (d^{j+3}P_u -d^{j+3}P_v )( \nabla^{m_1 + 1}u, \dots, \nabla^{m_j +1}u)( \nabla^{k_1 +1} u \star \nabla^{k_2 +1} u \star \nabla^{k_3 +1} u\star \nabla^{k_4 +1} u)\\ \nonumber
	&\quad+ \sum_{(**)} d^{j+3}P_v ( \nabla^{m_1 + 1}w, \nabla^{m_2 + 1}u,\dots, \nabla^{m_j +1}u)( \nabla^{k_1 +1} u \star \nabla^{k_2 +1} u \star \nabla^{k_3 +1} u\star \nabla^{k_4 +1} u)\\
	&\qquad \colon~~\\
	&\quad+ \sum_{(**)} d^{j+3}P_v ( \nabla^{m_1 + 1}v,\dots, \nabla^{m_{j-1} +1}v,\nabla^{m_{j} +1}w )( \nabla^{k_1 +1} u \star \nabla^{k_2 +1} u \star \nabla^{k_3 +1} u\star \nabla^{k_4 +1} u)
	\end{align*}
where we sum over $ (*) : j = 1,  \dots ,m~\text{and}~ m_1 + \dots + m_j + k_1 + k_2 + k_3 + k_4 = m - j,~ (**) : j = 2,  \dots ,m~\text{and}~ m_1 + \dots + m_j + k_1 + k_2 + k_3 + k_4 = m - j.$\\[5pt]
Also, the case $m = 1 $ is similar.
\end{Corollary}
\begin{proof}
The assertions are consequences of the Leibniz rule and  Lemma \ref{uniqueness-lemma1}.
\end{proof}

\begin{Corollary}\label{uniqueness-corollary2}
	We have for $ m \in \N,~m \geq 2$ and $ w = u-v$ that
	$$ \nabla^m ( \mathcal{N}(u) - \mathcal{N}(v)) $$
	is a linear combination of terms of the form
	\begin{align*}
	&(d^{j+1}P_u -d^{j+1}P_v )( \nabla^{m_1 + 1}u, \dots, \nabla^{m_j +1}u)( \nabla^{k_1}u_t \star \nabla^{k_2} u_t\\
	&\hspace{6cm}+ \nabla^{k_1+2}u \star \nabla^{k_2+2} u + \nabla^{k_1+3}u \star \nabla^{k_2} u),\\[4pt] 
	&d^{j+1}P_v ( \nabla^{m_1 + 1}w, \nabla^{m_2 + 1}h_1,\dots, \nabla^{m_j +1}h_{j-1})( \nabla^{k_1}u_t \star \nabla^{k_2} u_t\\
	&\hspace{6cm} + \nabla^{k_1+2}u \star \nabla^{k_2+2} u+ \nabla^{k_1+3}u \star \nabla^{k_2} u),\\[10pt]
	&(d^{j+2}P_u -d^{j+2}P_v )( \nabla^{m_1 + 1}u, \dots, \nabla^{m_j +1}u)( \nabla^{k_1+1}u \star \nabla^{k_2+1} u \star \nabla^{k_3+2}u),\\[4pt] 
	&d^{j+2}P_v ( \nabla^{m_1 + 1}w, \nabla^{m_2 + 1}h_1,\dots, \nabla^{m_j +1}h_{j-1})( \nabla^{k_1+1}u \star \nabla^{k_2+1} u \star \nabla^{k_3+2}u),\\[10pt]
	&(d^{j+3}P_u -d^{j+3}P_v )( \nabla^{m_1 + 1}u, \dots, \nabla^{m_j +1}u)( \nabla^{k_1 +1} u \star \nabla^{k_2 +1} u \star \nabla^{k_3 +1} u\star \nabla^{k_4 +1} u),\\[4pt] 
	&d^{j+3}P_v ( \nabla^{m_1 + 1}w, \nabla^{m_2 + 1}h_1,\dots, \nabla^{m_j +1}h_{j-1})( \nabla^{k_1 +1} u \star \nabla^{k_2 +1} u \star \nabla^{k_3 +1} u\star \nabla^{k_4 +1} u),~~~\text{and}\\[14pt]
	&d^{j+1}P_v(\nabla^{m_1 + 1}v, \dots, \nabla^{m_j +1}v)(\nabla^{k_1}w_t \star \nabla^{k_2}h_t  + \nabla^{k_1+2}w \star \nabla^{k_2+2} h\\ 
	& ~~~ + \nabla^{k_1+3}w \star \nabla^{k_2} h + \nabla^{k_1+3}h \star \nabla^{k_2} w ),~~ h \in \{ u, v\},\\[10pt]
	&~~d^{j+2}P_v(\nabla^{m_1 + 1}v, \dots, \nabla^{m_j +1}v)(\nabla^{k_1+1}w \star \nabla^{k_2+1}h_1 \star \nabla^{k_3 +2}h_2\\
	& \hspace{6cm}  + \nabla^{k_1+1}h_1 \star \nabla^{k_2+1}h_2 \star \nabla^{k_3 +2}w),\\[10pt]
	&~~d^{j+3}P_v(\nabla^{m_1 + 1}v, \dots, \nabla^{m_j +1}v)(\nabla^{k_1+1}w \star \nabla^{k_2+1}h_1 \star \nabla^{k_3 +1}h_2 \star \nabla^{k_4 +1}h_3),
	\end{align*}
	where $ j, k_1, k_2, k_3, k_4 $,~$ m_1, \dots m_j$ and $h, h_1, \dots, h_{j-1} \in \{ u, v\} $ are as above in Corollary \ref{uniqueness-corollary1}. 
	Also, we have a similar (but simpler) statement for $ m = 1$.
\end{Corollary}
\begin{proof}
	We write, according to the definition of $ \mathcal{N}(u) $ in \eqref{expansion},
	\begin{align*}
	\mathcal{N}(u) - \mathcal{N}(v)& = (dP_u - dP_v)(u_t \cdot u_ t + \nabla^2 u \star \nabla^2 u +  \nabla^3 u \star \nabla u)\\
	& + (d^2P_u - d^2 P_v)(\nabla u \star \nabla u \star \nabla^2u) +  (d^3P_u - d^3P_v) (\nabla u \star \nabla u \star \nabla u \star \nabla u)\\
	& + dP_v(w_t \cdot u_t + v_t \cdot w_t + \nabla w \star \nabla u + \nabla v \star \nabla w + \nabla^3 w \star \nabla u + \nabla^3 v \star \nabla w)\\
	& + d^2P_v(\nabla w \star \nabla u \star \nabla^2 u + \nabla v \star \nabla w \star \nabla^2 u + \nabla v \star \nabla v \star \nabla^2 w)\\
	& + d^3P_v(\nabla w \star \nabla u \star \nabla u \star \nabla u + \nabla v \star \nabla w \star \nabla u \star \nabla u\\
	&~~~~~ +  \nabla v \star \nabla v \star \nabla w \star \nabla u + \nabla v \star \nabla v \star \nabla v \star \nabla w).
	\end{align*}
	Then, we use Corollary \ref{uniqueness-corollary1} for the first three terms in the sum above. For the latter three, we use Lemma \ref{Lemma-Leibniz} and the Leibniz rule.
      \end{proof}
 
 Let $ \varepsilon \in  (0,1)$. We recall from \eqref{Reg-expansion}  the  definition
\begin{align*}
\mathcal{N}_{\varepsilon}(u) =\mathcal{N}(u) - \varepsilon d^2 P_u ( u_t, \nabla u, \nabla u) - \varepsilon 2 dP_u(\nabla u_t, \nabla u) - \varepsilon dP_u( u_t, \Delta u).
\end{align*} 
\begin{Lemma}\label{remark-parabolic-nonlin}
For $ m \in \N_0$ the derivative $ \nabla^m(\mathcal{N}_{\varepsilon}(u))$  compared to $ \nabla^m(\mathcal{N}(u)) $ contains the  additional terms
\begin{align*}
&d^{j+1}P_u (\nabla^{m_1+1} u \star \dots \star \nabla^{m_j +  1}u )( \nabla^{k_1 } u_t \star \nabla^{k_2 + 2} u 
  + \nabla^{k_1 + 1} u_t  \star \nabla^{k_2+1} u ),~~\text{and}\\
& d^{j+2}P_u (\nabla^{m_1+1} u \star \dots \star \nabla^{m_j +  1}u )(  \nabla^{k_1 } u_t \star \nabla^{k_2 + 1} u \star \nabla^{k_3 + 1} u ),
\end{align*} 
with $ j,m_1, \dots, m_j,k_1,k_2,k_3$ similarly to Lemma \ref{diff-Lemma}.

Further $ \nabla^m(\mathcal{N}_{\varepsilon}(u)) - \nabla^m(\mathcal{N}_{\varepsilon}(v))$ compared to $ \nabla^m(\mathcal{N}(u)) - \nabla^m(\mathcal{N}(v)) $ 
contains additional terms of the form
\begin{align*}
&(d^{j+1}P_u -d^{j+1}P_v )( \nabla^{m_1 + 1}u, \dots, \nabla^{m_j +1}u)( \nabla^{k_1}u_t \star \nabla^{k_2+2} u + \nabla^{k_1+1}u_t \star \nabla^{k_2+1} u),\\[4pt]
&d^{j+1}P_v ( \nabla^{m_1 + 1}w, \nabla^{m_2 + 1}h_1,\dots, \nabla^{m_j +1}h_{j-1})( \nabla^{k_1}u_t \star \nabla^{k_2+2} u + \nabla^{k_1+1}u_t \star \nabla^{k_2+1} u),\\[10pt]
&(d^{j+2}P_u -d^{j+2}P_v )( \nabla^{m_1 + 1}u, \dots, \nabla^{m_j +1}u)( \nabla^{k_1}u_t \star \nabla^{k_2+1} u \star \nabla^{k_3+1}u),\\[4pt]
&d^{j+2}P_v ( \nabla^{m_1 + 1}w, \nabla^{m_2 + 1}h_1,\dots, \nabla^{m_j +1}h_{j-1})( \nabla^{k_1}u_t \star \nabla^{k_2+1} u \star \nabla^{k_3+1}u),~~~\text{and}\\[10pt]
&d^{j+1}P_v(\nabla^{m_1 + 1}v, \dots, \nabla^{m_j +1}v)(\nabla^{k_1}w_t \star \nabla^{k_2+2}h  + \nabla^{k_1+1}w_t \star \nabla^{k_2+1} h\\
&~~~ + \nabla^{k_1}h \star \nabla^{k_2+2} w + \nabla^{k_1+1}h_t \star \nabla^{k_2+1} w ),~~ h \in \{ u, v\},\\[10pt]
&~~d^{j+2}P_v(\nabla^{m_1 + 1}v, \dots, \nabla^{m_j +1}v)(\nabla^{k_1}w_t \star \nabla^{k_2+1}h_1 \star \nabla^{k_3 +1}h_2\\
&\hspace{6cm}  + \nabla^{k_1}(h_1)_t \star \nabla^{k_2+1}h_2 \star \nabla^{k_3 +1}w),
\end{align*}
with  $ w = u-v $ and  $ j, m_1,\dots, m_j, k_1,k_2, k_3, h_1, \dots, h_{j-1} $  similarly to Corollary \ref{uniqueness-corollary2}.
\end{Lemma}
The implicit constants may depend on $\varepsilon$ here.
\section{Approximation of the initial data}\label{section-appendix2}
In this section we construct certain approximations of initial data in order to conclude continuous dependence of the solution 
on the initial data. As in the previous sections, take functions $ u_0 , u_1 : \R^n \to \R^L$ with $ u_0 \in N$,  $u_1 \in T_{u_0}N $ a.e. on $ \R^n $, and
$$ (\nabla u_0 , u_1 ) \in H^{k-1}(\R^n) \times H^{k-2}(\R^n).$$
 for some $ k > \lfloor \f{n}{2} \rfloor +2$ with $ k \in \N $.
\begin{Lemma}\label{approximation-lemma}
	Let the functions $ (u_0, u_1) $ be as above. Then there is a number $\delta^*=\delta^*(N)>0$ such that for
	$ \delta \in( 0,\delta*] $ there exist maps $ u_0^{\delta}, u_1^{\delta} \in C^{\infty}( \R^n, \R^L) $ such that  $ \nabla u_0^{\delta}, u_1^{\delta} \in H^m$ for all $m\in\N$,
	$ u_0^{\delta} \in N$ and $u_1^{\delta} \in T_{u_0^{\delta}}N$ on $ \R^n$ which satisfy
	\begin{align}
	 u_0 - u_0^{\delta} \in L^2\quad \text{and}\quad \norm{u_0 - u_0^{\delta}}{L^2} &\le C_0 \delta,\label{nullapp}\\
	\norm{(\nabla u_0^{\delta}, u_1^{\delta})- (\nabla u_0, u_1)}{H^{k-2}\times H^{k-3}} &= o(\sqrt{\delta})\quad \text{as }~\delta \to 0^+, \label{einsapp}\\
	\norm{(\nabla u_0^{\delta}, u_1^{\delta})- (\nabla u_0, u_1)}{H^{k-1}\times H^{k-2}} &= o(1) \quad\text{as}~\delta \to 0^+, \label{zweiapp}\\
	\norm{(\nabla u_0^{\delta}, u_1^{\delta})}{H^{k}\times H^{k-1}} &\le C_0 \f{1}{\sqrt{\delta}}\label{dreiapp}
	\end{align}
	for a constant $ C_0 = C_0(\norm{P_p}{C^k_b}, \norm{\nabla u_0}{H^{k-1}}, \norm{u_1}{H^{k-2}}) \ge1$.
	Further let $ (v_0, v_1) $ be as above with $  u_0- v_0 \in H^k(\R^n)$ and
	$$ \norm{(u_0, u_1) - (v_0, v_1)}{H^{k}\times H^{k-2}}\leq R$$
	for some $R>0$. Then for $ \delta \in( 0,\delta^*] $ we have
	\begin{align}
	\norm{(\nabla v_0^{\delta}, v_1^{\delta})}{H^{k}\times H^{k-1}} &\leq C_0( 1 + R^k) \f{1}{\sqrt{\delta}},\label{vierapp}\\
	\norm{(u_0^{\delta}, u_1^{\delta}) -  (v_0^{\delta}, v_1^{\delta})}{H^k\times H^{k-2}} 
	&\leq C_0 (1 + R^k) \norm{(u_0, u_1) -  (v_0, v_1)}{H^k\times H^{k-2}}.\label{fünfapp}
	\end{align}
\end{Lemma}
\begin{proof}
We choose the caloric extension for regularization, i.e., we consider $ \eta_{\delta} * u_0$ and $\eta_{\delta} * u_1 $ where
	$$ \eta_{\delta}(x) = (4\pi\delta)^{-\f{n}{2}} e^{- \f{|x|^2}{4\delta }},\quad \delta > 0,~~ x \in \R^n,$$
and  $T(\delta)f = \eta_{\delta}*f$	is the heat semigroup. Since $ u_1 \in C^0_b(\R^n)$ and $ u_0 \in C^2_b(\R^n) $ by assumption, the convolution 
is well defined for $ u_0 $ and $ u_1 $. Moreover, $ \eta_{\delta} * u_0$ tends to $u_0$ and $\eta_{\delta} * u_1$ to $u_1 $ uniformly as $ \delta \to 0^+$, as well as
$$  \nabla(\eta_{\delta} * u_0) \to \nabla u_0~~\text{ in}~ H^{k-1}(\R^n),\qquad \eta_{\delta} * u_1 \to u_1 ~~\text{in}~H^{k-2}(\R^n)\qquad\text{as}~\delta \to 0^+.$$
The uniform convergence yields
	\begin{align}
	\dist(u_0*\eta_{\delta}(x) , N) \leq  | u_0*\eta_{\delta}(x) - u_0(x) | \to 0\qquad \text{as }~\delta \to 0^+
	\end{align}
	uniformly in $ x \in \R^n$. Hence, if $\delta > 0 $ is small enough we can define 
	$$ u_0^{\delta} : = \pi(u_0 * \eta_{\delta})\quad\text{ and }\quad u_1^{\delta} : = P_{u_0* \eta_{\delta}}(u_1 * \eta_{\delta}).$$
	Recall that $ \pi$ is the nearest point map and that $ P_{u_0* \eta_{\delta}}(u_1 * \eta_{\delta}) \in T_{u_0^{\delta}}N $ by definition of the projector $P$ and $ u_0^{\delta}$.  Especially we have
	\begin{align*}
	| u_0^{\delta}(x) - u_0 * \eta_{\delta}(x) | &= \dist( u_0* \eta_{\delta}(x), N)\leq | u_0(x) - u_0 *\eta_{\delta}(x) |, \\ 
	 | u_0^{\delta}(x) - u_0(x) | &\leq 2| u_0(x) - u_0 *\eta_{\delta}(x) |
	 \end{align*}
for $x \in \R^n.$	We further note that $ u_0^{\delta}$ and $u_1^{\delta} $ are smooth maps and that we have the uniform convergence
	$$ u_0^{\delta} \to u_0, \qquad  u_1^{\delta} \to u_1$$
	as $\delta \to 0^+$ by construction of $ u_0^{\delta} $ (and the mean value theorem for $ u_1^{\delta}$).
	Assertion  \eqref{nullapp} follows from
	$$ \norm{\delta^{-1}( u_0 *\eta_{\delta} - u_0) }{L^2 } = \norm{\f{1}{\delta} \int_0^{\delta}(\Delta u_0) *\eta_{s}\,ds }{L^2} \lesssim \norm{\Delta u_0 }{L^2},$$
	by Young's inequality for the convolution. Since $ \nabla u_0^{\delta} = P_{u_0*\eta_{\delta}}((\nabla u_0) * \eta_{\delta})$, we further
	have to treat the  terms 
	\begin{align*}
	P_{u_0*\eta_{\delta}}((\nabla u_0) * \eta_{\delta}) - \nabla u_0 &= P_{u_0*\eta_{\delta}}((\nabla u_0) * \eta_{\delta} - \nabla u_0 ) 
	 + (P_{u_0*\eta_{\delta}} - P_{u_0}) \nabla u_0,\\
	P_{u_0*\eta_{\delta}}(u_1 * \eta_{\delta}) - u_1 &= P_{u_0*\eta_{\delta}}( u_1 * \eta_{\delta} - u_1 ) + (P_{u_0*\eta_{\delta}} - P_{u_0}) u_1.
	\end{align*}
	We start by estimating (by means of the mean value theorem for $P$)
	\begin{align*}
	\norm{P_{u_0*\eta_{\delta}}((\nabla u_0) * \eta_{\delta}) - \nabla u_0}{L^2} &\leq \norm{P_{u_0*\eta_{\delta}}((\nabla u_0) * \eta_{\delta} - \nabla u_0 )}{L^2} + \norm{(P_{u_0*\eta_{\delta}} - P_{u_0}) \nabla u_0}{L^2}\\
	&\lesssim  \delta\left( O(1) + \norm{\nabla u_0}{L^2} \norm{\f{1}{\delta}( u_0 * \eta_{\delta} - u_0)}{L^{\infty}}\right), 
	\end{align*}
	where $ \f{1}{\delta}( u_0 * \eta_{\delta} - u_0) \to \Delta u_0 $ uniformly as $ \delta \to 0^+$ since $ u_0 \in C^2_b(\R^n)$. 
	Similarly, employing Lemmas \ref{Lemma-Leibniz}, \ref{Moser-estimate} and \ref{uniqueness-lemma1} as before, we see
	\begin{align*}
	&\norm{\nabla^{k-2}(P_{u_0*\eta_{\delta}}((\nabla u_0) * \eta_{\delta}) - \nabla u_0)}{L^2}\\
	&\lesssim \sum_{l_1 + l_2 = k-2}\!\Big[\norm{\nabla^{l_1}(P_{u_0*\eta_{\delta}}) \cdot \nabla^{l_2}((\nabla u_0) * \eta_{\delta} - \nabla u_0 )}{L^2} 
	 + \norm{\nabla^{l_1}(P_{u_0*\eta_{\delta}} - P_{u_0}) \cdot \nabla^{l_2 + 1} u_0}{L^2}\Big]\\
	&\lesssim ( 1+ \norm{\nabla u_0}{H^{k-2}}^{k} + \norm{(\nabla u_0) * \eta_{\delta} }{H^{k-2}}^{k}) \norm{(\nabla u_0) * \eta_{\delta} - \nabla u_0 }{H^{k-2}}\\
	&\quad + \delta \norm{\nabla^{k-1}u_0}{L^2}^k\norm{\delta^{-1}( u_0*\eta_{\delta} - u_0)}{L^{\infty}}\\
	&\lesssim o(\sqrt{\delta})\qquad\text{as }~ \delta \to 0^+.
	\end{align*}
	Here we also use \cite[Prop. 2.2.4]{lunardisemigroup}. Interpolation and an analogous argument for $ u_1^{\delta}$ in $H^{k-3}$ then allows us  
	to conclude  \eqref{einsapp}. Assertion \eqref{zweiapp} is shown in the same way, with $ o(1)$ instead of  $ o(\sqrt{\delta}) $ in the upper bound.
	For \eqref{dreiapp}, we compute 
	\begin{align*}
	&\norm{\nabla^{k}(P_{u_0*\eta_{\delta}}((\nabla u_0) * \eta_{\delta}))}{L^2}\\ 
	&\quad\lesssim \sum_{\substack{l_1 + l_2 = k\\ l_1 > 0 }}\norm{\nabla^{l_1}(P_{u_0*\eta_{\delta}}) \cdot( \nabla^{l_2+1} u_0 * \eta_{\delta})}{L^2} 
	   + \norm{P_{u_0*\eta_{\delta}} \nabla ( \nabla^{k} u_0 * \eta_{\delta})}{L^2}\\
	&\quad \lesssim (1 + \norm{\nabla u_0}{H^{k-1}}^k)\norm{\nabla u_0}{H^{k-1}} + \norm{P_{u_0*\eta_{\delta}} \nabla ( \nabla^{k} u_0 * \eta_{\delta})}{L^2}
	\end{align*}
	as before. The last  term is bounded via
	$$ \norm{P_{u_0*\eta_{\delta}} \nabla ( \nabla^{k} u_0 * \eta_{\delta})}{L^2} 
	   \lesssim \norm{ (\nabla^{k} u_0) * \nabla(\eta_{\delta})}{L^2} \lesssim \f{1}{\sqrt{\delta}}\norm{\nabla u_0}{H^{k-1}}$$
	again by Young's inequality. Similarly, the term  $\nabla^{k-1}
        u_1^{\delta}$ is estimated in $L^2( \R^n)$. The above reasoning also shows \eqref{vierapp} 
	if we choose the constant $ C_0 > 0$ suitably. In order to prove \eqref{fünfapp}, similarly as above 
	we compute
	$$ \norm{u_0^ {\delta}- v_0^{\delta}}{L^2} \lesssim \norm{\eta_{\delta}*( u_0 - v_0)}{L^2} \lesssim \norm{u_0 - v_0}{L^2}.$$
	by the mean value theorem and Young's inequality. Writing
	\begin{align*}
	P_{u_0*\eta_{\delta}}((\nabla u_0)& * \eta_{\delta}) - P_{v_0*\eta_{\delta}}((\nabla v_0) * \eta_{\delta})\\
 & = P_{u_0*\eta_{\delta}}((\nabla u_0) * \eta_{\delta} - (\nabla v_0) * \eta_{\delta})  + (P_{u_0*\eta_{\delta}}- P_{v_0*\eta_{\delta}})((\nabla v_0) * \eta_{\delta}),
	\end{align*}
	we deduce
	\begin{align*}
	\|\nabla^{k-1}&(P_{u_0*\eta_{\delta}}((\nabla u_0) * \eta_{\delta}) - P_{v_0*\eta_{\delta}}((\nabla v_0) * \eta_{\delta}))\|_{L^2}\\
	&\lesssim \sum_{l_1 + l_2 = k-1}\norm{\nabla^{l_1}(P_{u_0*\eta_{\delta}}) \cdot \nabla^{l_2}((\nabla u_0) * \eta_{\delta} - (\nabla v_0)*\eta_{\delta} )}{L^2}\\
	&\quad+ \sum_{l_1 + l_2 = k-1}\norm{\nabla^{l_1}(P_{u_0*\eta_{\delta}}- P_{v_0*\eta_{\delta}}) \cdot (\nabla^{l_2 + 1} v_0) * \eta_{\delta}}{L^2}\\
	&\lesssim ( 1+ \norm{\nabla u_0}{H^{k-1}}^{k} + \norm{\nabla v_0  }{H^{k-1}}^{k}) \norm{\nabla u_0 - \nabla v_0 }{H^{k-1}} 
	+  \norm{\nabla^{k}v_0}{L^2}^k\norm{ u_0 - v_0}{L^{\infty}}\\
	&\lesssim ( 1 + \norm{\nabla u_0}{H^{k-1}}^{k} + R^k) \norm{\nabla u_0 - \nabla v_0 }{H^{k-1}} + \norm{\nabla^{k}v_0}{L^2}^k\norm{ u_0 - v_0}{H^{k}}\\
	&\lesssim ( 1 + \norm{\nabla u_0}{H^{k-1}}^{k} + R^k) \norm{\nabla u_0 - \nabla v_0 }{H^{k-1}},
	\end{align*}
	The claim \eqref{fünfapp} then follows by interpolation and a proper choice of $ C_0 > 0 $. Finally the estimate for
	$$ u_1^{\delta} - v_1^{\delta} = P_{u_0*\eta_{\delta}}(u_1 * \eta_{\delta} - v_1 * \eta_{\delta})  
	   + (P_{u_0*\eta_{\delta}}- P_{v_0*\eta_{\delta}})(v_1 * \eta_{\delta})$$
	works similarly.
\end{proof}

\section{Establishing the identity \eqref{identity-final}}\label{appendix-identity-final}
For $ f,g \in H^1(\R^n)$, $ h \in \R $ and $ i \in \{ 1, \dots , n\}$ we set
$$ D_h^if(x) = \f{1}{h}( f(x + e_ih ) - f(x)). $$
Observe that $ D_h^i(fg)(x) = (D_h^if)(x) g(x + e_ih) + f(x) (D_h^ig)(x)$. Since we only use the product rule integrated over 
$ x \in \R^n $ and $ g( \cdot + he_i) \to g $ strongly in $H^1$ as $h \to 0$, we drop the $ h$-dependence in $ g(\cdot + e_i h)$ in the following calculation.

Let $u$ be the solution of \eqref{expansion} obtained in Section~\ref{main-theorem-proof}. We compute
\begin{align*}
&\f{1}{2}\f{d}{dt}\left( \norm{D_h^i\nabla^{k-3} u_t}{L^2}^2 + \norm{D_h^i\nabla^{k-1} u}{L^2}^2\right)
  = \int_{\R^n} D_h^i\nabla^{k-3} \Big((I- P_u) \mathcal{N}(u) \Big)\cdot D_h^i\nabla^{k-3} u_t \,dx\\
&= \sum_{l=1}^{k-3}\int_{\R^n}D_h^i( \nabla^l(I-P_u) \star\nabla^{k-3-l}\mathcal{N}(u)) \cdot D_h^i\nabla^{k-3}u_t\,dx \\
 &\quad  + \int_{\R^n}D_h^i(I-P_u) \nabla^{k-3}\mathcal{N}(u) \cdot D_h^i\nabla^{k-3}u_t\,dx
  + \int_{\R^n} D_h^i\nabla^{k-3}\mathcal{N}(u) \cdot (I-P_u) D_h^i\nabla^{k-3}u_t\,dx\\
&= \sum_{l=1}^{k-3}\int_{\R^n}D_h^i( \nabla^l(I-P_u) \star\nabla^{k-3-l}\mathcal{N}(u)) \cdot D_h^i\nabla^{k-3}u_t\,dx \\
&\quad  + \int_{\R^n} \! D_h^i(I-P_u) \nabla^{k-3}\mathcal{N}(u) \cdot D_h^i\nabla^{k-3}u_t\,dx
 + \int_{\R^n} \! \! D_h^i(\nabla^{k-3}\mathcal{N}(u) \cdot (I-P_u) D_h^i\nabla^{k-3}u_t)\,dx \\
&\quad  + \int_{\R^n} \nabla^{k-3}\mathcal{N}(u) \cdot D_h^i(D_h^i(I-P_u) \nabla^{k-3}u_t)\,dx\\
&\quad + \sum_{l=1}^{k-3} \int_{\R^n} \nabla^{k-3}\mathcal{N}(u) \cdot (D_h^i)^2(\nabla^{l}(I-P_u)\star \nabla^{k-3-l}u_t)\,dx = : \int_{\R^n} T_h^i(u)\,dx,
\end{align*}
where the second identity follows from $ (I - P_u) u_t = 0$.
For a fixed time $ t \in [0, T_m)$, the regularity of $u$ yields the limit
\begin{align*}
\lim_{h \to 0} \int_{\R^n} T_h^i(u(t))\,dx 
&=  \sum_{l=1}^{k-3}\int_{\R^n}\partial_{x_i}( \nabla^l(I-P_u) \star \nabla^{k-3-l}\mathcal{N}(u)) \cdot \nabla^{k-3}\partial_{x_i}u_t\,dx\\ 
&\quad-\int_{\R^n}dP_u(\partial_{x_i}u,\nabla^{k-3}\mathcal{N}(u)) \cdot \nabla^{k-3}\partial_{x_i}u_t\,dx\\
&\quad - \int_{\R^n} \nabla^{k-3}\mathcal{N}(u) \cdot \partial_{x_i}\big(dP_u(\partial_{x_i}u,\nabla^{k-3}u_t)\big)\,dx\\
&\quad + \sum_{l=1}^{k-3} \int_{\R^n} \nabla^{k-3}\mathcal{N}(u) \cdot \partial_{x_i}^2(\nabla^{l}(I-P_u)\star \nabla^{k-3-l}u_t)\,dx\\
& =:  \int_{\R^n} T^i(u(t))\,dx.
\end{align*}
Here we also used that
$$  \int_{\R^n} D_h^i(\nabla^{k-3}\mathcal{N}(u) \cdot (I-P_u) D_h^i\nabla^{k-3}u_t)\,dx  \to 0 \quad \text{as}~~h \to 0 $$
by Gauss' Theorem. Estimating as in Section \ref{uniform-energy}, we derive 
\begin{align*}
\left| \int_{\R^n} T^i(u(t))\,dx\right|  \lesssim \sup_{s \in [0,T]}( 1 + \norm{\nabla u(s)}{H^{k-1}}^{2k} + \norm{u_t(s)}{H^{k-2}}^{2k}  )
    (\norm{\nabla u(s)}{H^{k-1}}^{2} + \norm{u_t(s)}{H^{k-2}}^{2}  ).
\end{align*}
for $ t \in [0,T]$ and $T<T_m$. In the limit $ h \to 0$ it follows 
\begin{align*}
\norm{\nabla^{k-3} \partial_{x_i}\!u_t}{L^2}^2 + \norm{\nabla^{k-1} \partial_{x_i}\!u}{L^2}^2 
   &=  2\! \!\int_0^t\!\!\int_{\R^n}\!\!\! T^i(u(s))\,dx\,ds + \norm{\nabla^{k-3} \partial_{x_i}\!u_1}{L^2}^2 + \norm{\nabla^{k-1} \partial_{x_i}\!u_0}{L^2}^2
\end{align*}
by dominated convergence. The right-hand side is continuous in $t$, and hence the highest derivatives $\nabla^{k} u_t, \nabla^{k-2}u:[0, T_m)\to L^2$ 
are continuous, since we already know their weak continuity.
Finally, summing over $i =1, \dots, n$ and estimating $ T^i(u) $ as in Section \ref{uniform-energy}, 
we conclude the blow-up criterion from \eqref{blowup} for the solution $u$.

\bibliographystyle{amsplain}

\begin{thebibliography}{1}

\bibitem{bejenaru_et_al}
Ioan Bejenaru, Alexandru D. Ionescu, Carlos E. Kenig and Daniel Tataru.
\newblock Global {S}chr\"{o}dinger maps in dimensions {$d\geq 2$}: small
data in the critical {S}obolev spaces.
\newblock {\em Ann. of Math. (2)}, 173(3): 1443--1506, 2011.

\bibitem{courant_hilbert}
Richard Courant and David Hilbert.
\newblock {\em Methods of mathematical physics. {V}ol. {I}.}
\newblock Interscience Publishers, Inc., New York, N.Y., 1953.



\bibitem{denkschnaubelt}
Robert Denk and Roland Schnaubelt.
\newblock A structurally damped plate equation with {D}irichlet-{N}eumann
boundary conditions.
\newblock {\em J. Differential Equations}, 259(4):1323--1353, 2015.

\bibitem{erdogan_tzirakis}
M.~Burak Erdogan and Nikolaos Tzirakis.
\newblock {\em Dispersive Partial Differential Equations: Wellposedness and
	Applications}.
\newblock Cambridge University Press, 2016.

\bibitem{fan2010regularity}
Jishan Fan and Tohru Ozawa.
\newblock On regularity criterion for the 2{D} wave maps and the 4{D}
biharmonic wave maps.
\newblock In {\em Current advances in nonlinear analysis and related topics},
volume~32 of {\em GAKUTO Internat. Ser. Math. Sci. Appl.}, pages 69--83.
Gakkotosho, Tokyo, 2010.

\bibitem{LammSchnaubeltHerr}
Sebastian {Herr}, Tobias {Lamm}, and Roland {Schnaubelt},
\newblock Biharmonic wave maps into spheres.
\newblock to appear in {\em Proc. Amer. Math. Soc.}


\bibitem{kenig2010cauchy}
Carlos Kenig, Tobias Lamm, Daniel Pollack, Gigliola Staffilani, and Tatiana
Toro.
\newblock The {C}auchy problem for {S}chr\"{o}dinger flows into {K}\"{a}hler
manifolds.
\newblock {\em Discrete Contin. Dyn. Syst.}, 27(2):389--439, 2010.

\bibitem{lasieckacontrol}
Irena Lasiecka and Roberto Triggiani.
\newblock {\em Control theory for partial differential equations: continuous
	and approximation theories. {I}}.
\newblock Cambridge University Press, Cambridge, 2000.

\bibitem{lunardisemigroup}
Alessandra Lunardi.
\newblock {\em Analytic semigroups and optimal regularity in parabolic}.
\newblock Birkh\"{a}user/ Springer Basel AG, Basel, 1995.

\bibitem{lunardiinterpolation}
Alessandra Lunardi.
\newblock {\em Interpolation theory}.
\newblock Edizioni della Normale, Pisa, 2018.
\newblock Third edition.

\bibitem{nahmod_et_al}
Andrea Nahmod, Atanas Stefanov and Karen Uhlenbeck.
\newblock On {S}chr\"{o}dinger maps.
\newblock {\em Comm. Pure Appl. Math.}, 56(1): 114--151, 2003.

\bibitem{pausader}
Beno\^{\i}t Pausader.
\newblock Scattering and the {L}evandosky-{S}trauss conjecture for
fourth-order nonlinear wave equations.
\newblock {\em J. Differential Equations}, 241(2): 237--278, 2007

\bibitem{schmid}
Tobias Schmid.
\newblock Energy bounds for biharmonic wave maps in low dimensions.
\newblock {\em CRC 1173-Preprint} 2018/51, Karlsruhe Institute of Technology, 2018.


\bibitem{shatahstruwe}
Jalal Shatah and Michael Struwe.
\newblock {\em Geometric wave equations}.
\newblock American Mathematical Society, Providence, RI, 1998.


\bibitem{sogge}
Christopher~D. Sogge.
\newblock {\em Lectures on non-linear wave equations}.
\newblock International Press, Boston, MA, second edition, 2008.


\bibitem{sterbenz-tataru}
Jacob Sterbenz and Daniel Tataru.
\newblock Regularity of wave-maps in dimension {$2+1$}.
\newblock {\em Comm. Math. Phys.}, 298(1): 231--264, 2010.

\bibitem{tao1}
Terence Tao.
\newblock Global regularity of wave maps. {I}. {S}mall critical
{S}obolev norm in high dimension.
\newblock {\em Internat. Math. Res. Notices}, No. 6: 299--328, 2001.


\bibitem{tao2}
Terence Tao.
\newblock Global regularity of wave maps. {II}. {S}mall energy in two
dimensions.
\newblock {\em Comm. Math. Phys.}, 224(2): 443--544, 2001.

\bibitem{tataru}
Daniel Tataru.
\newblock The wave maps equation.
\newblock {\em Bull. Amer. Math. Soc. (N.S.)}, 41(2): 185--204, 2004.

\bibitem{taylor3}
Michael~E. Taylor.
\newblock {\em Partial differential equations {III}. {N}onlinear equations}.
\newblock Springer, New York, second edition, 2011.

\bibitem{Tzvetkov}
Nikolay Tzvetkov.
\newblock Ill-posedness issues for nonlinear dispersive equations.
\newblock In {\em Lectures on nonlinear dispersive equations}, volume~27 of
{\em GAKUTO Internat. Ser. Math. Sci. Appl.}, pages 63--103. Gakkotosho,
Tokyo, 2006.

\end{thebibliography}

\end{document}